\documentclass[journal]{IEEEtran}

\ifCLASSINFOpdf
\usepackage[pdftex]{graphicx}
\graphicspath{{Figs/}}
\usepackage{comment}
\usepackage[]{algorithm2e}
\else
\fi
\usepackage{mathptmx} 
\usepackage{times} 
\usepackage{amsmath,amsfonts,amssymb,amscd,amsthm,xspace}
\usepackage{algorithmic}
\usepackage{fixltx2e}

\usepackage{epsfig}
\usepackage{enumitem}
\usepackage{bbm}
\usepackage{dblfloatfix}
\usepackage{bm}
\usepackage{booktabs}
\usepackage{color}
\usepackage{graphicx}
\graphicspath{{Figs/}}

\newtheorem{theorem}{\textit{Theorem}}
\numberwithin{theorem}{section}
\newtheorem{lemma}{\textit{Lemma}}
\newtheorem{corollary}{\textit{Corollary}}
\numberwithin{corollary}{section}

\newtheorem{assumption}{\textit{Assumption}}
\newtheorem{example}{\textit{Example}}
\newtheorem*{aggregator}{\textit{Aggregator's Objective}}
\newtheorem*{operator}{\textit{Operator's Objective}}
\newtheorem*{system}{\textit{System's Goal}}
\newtheorem{definition}{\textit{Definition}}
\theoremstyle{definition}
\numberwithin{definition}{section}
\numberwithin{example}{section}

\newtheorem{remark}{\textit{Remark}}

\newcommand{\nn}{N}
\newcommand{\ns}{W}
\newcommand{\nt}{T}
\newcommand{\nm}{m}

\DeclareMathOperator*{\argmin}{arg\,min}

\DeclareMathOperator*{\arginf}{arg\,inf}

\begin{document}
%


\title{ Learning-Based Predictive Control via\\
Real-Time Aggregate Flexibility}
%
%
%

\author{Tongxin~Li,
        Bo~Sun,  
        Yue~Chen, 
        Zixin~Ye,
        Steven~H.~Low,
        and Adam~Wierman
\thanks{Tongxin Li and Steven Low acknowledge the support received from National Science Foundation
(NSF) through grants CCF 1637598, ECCS 1931662 and CPS ECCS 1932611. Bo Sun is supported
by Hong Kong Research Grant Council (RGC) General Research Fund (Project 16207318). Adam
Wierman’s research is funded by NSF (AitF-1637598 and CNS-1518941), Amazon AWS and VMware.}
\thanks{Li, Low and Wierman are with the Computing + Mathematical Sciences Department, California Institute of Technology, Pasadena, CA 91125 USA (e-mails: \{tongxin, slow, adamw\}@caltech.edu)}
\thanks{Sun is with the Department of Electronic and Computer Engineering, The Hong Kong University of Science and Technology, Hong Kong (e-mail: bsunaa@connect.ust.hk)}
\thanks{Chen is with the Department of Mechanical and Automation Engineering, the Chinese University of Hong Kong, Hong Kong SAR, China. (e-mail: yuechen@mae.cuhk.edu.hk)}
\thanks{Ye is with the Electrical Engineering Department, California Institute of Technology, Pasadena, CA 91125 USA (e-mail: zyye@caltech.edu)}}

\maketitle


\begin{abstract}

Aggregators have emerged as crucial tools for the coordination of distributed, controllable loads. To be used effectively, an aggregator must be able to communicate the available flexibility of the loads they control, as known as the aggregate flexibility to a system operator. However, most of existing aggregate flexibility measures often are slow-timescale estimations and much less attention has been paid to real-time coordination between an aggregator and an operator.
In this paper, we consider solving an online optimization in a closed-loop system and present a design of \textit{real-time} aggregate flexibility feedback, termed the \textit{maximum entropy feedback} (MEF). 
In addition to deriving analytic properties of the MEF,  {combining learning and control, we show that it can be approximated using reinforcement learning and used as a penalty term in a novel control algorithm -- the \textit{penalized predictive control} (PPC), which modifies vanilla model predictive control (MPC).} The benefits of { our scheme} are (1). \textit{Efficient Communication}. {An operator running PPC does not need to know the exact states and constraints of the loads, but only the MEF.} (2). \textit{Fast Computation}. {The PPC often has much less number of variables than an MPC formulation.} (3). \textit{Lower Costs.} {We show that under certain regularity assumptions, the PPC is optimal. We illustrate the efficacy of the PPC using a dataset from an adaptive electric vehicle charging network and show that PPC outperforms classical MPC.}
\end{abstract}


%

\IEEEpeerreviewmaketitle

\section{Introduction}
\label{sec:intro}
The uncertainty and volatility of renewable sources such as wind and solar power has created a need to exploit the flexibility of distributed energy resources (DERs) and aggregators have emerged as dominate players for coordinating these loads~\cite{callaway2010achieving,burger2017review}. 
An aggregator can coordinate a large pool of DERs and be a single point of contact for independent system operators (ISOs) to call on for flexibility.  
This enables ISOs to minimize cost, respond to unexpected fluctuations of renewables, and even mitigate failures quickly and reliably.  
Typically, an ISO communicates a time-varying signal to an aggregator, e.g., a desired power profile, that optimizes ISO objectives and the aggregator coordinates with the DERs to collectively respond to the time-varying signal as faithfully as possible, e.g., by shaping their aggregate power consumption to follow ISO's power profile, while satisfying DER constraints.
These constraints are often private to the loads, e.g., 
satisfying energy demands of electric vehicles before their deadlines. They limit the flexibility available to the aggregator so the aggregator must also communicate with the ISO by providing a feedback signal that
quantifies its available flexibility. This feedback provides ISO with crucial information for determining the signal it sends to the aggregator. Thus the aggregator and the ISO form a closed-loop control system to manage the aggregate flexibility of DERs.

This paper focuses on the design of this closed-loop system and, in particular, the design of real-time feedback signal from the aggregator to the ISO quantifying the available flexibility.  The design of \emph{the aggregate flexibility feedback}  signal is complex and has been the subject of significant research over the last decade, e.g.,~\cite{hao2014characterizing,hao2014aggregate,sajjad2016definitions,zhao2017geometric,madjidian2018energy,chen2018aggregating,sadeghianpourhamami2018quantitive,evans2019graphical,bernstein2016aggregation}.  Any feedback design must balance a variety of conflicting goals. 
Given the scale, complexity and privacy of the load constraints, it may neither be possible nor desirable to communicate precise information about every load.  Instead, aggregate flexibility feedback must be a concise summary of a system's constraints and it
must limit the leakage about specific load constraints. 
On the other hand, the feedback sent by an aggregator needs to be informative enough that it allows the ISO to achieve operational objectives, e.g., minimize cost, and, most importantly, {containing feasibility information of the whole system with respect to the private load constraints.} Moreover,
a design for a flexibility feedback signal must be general enough to be applicable for a wide variety of controllable loads, e.g., electric vehicles (EVs), heating, ventilation, and air conditioning (HVAC) systems,  energy storage units, thermostatically controlled loads, residential loads, and pool pumps. 
It is impractical to design different feedback signals for each load, so the same design must work for all DERs.
 
The challenge and importance of the design of flexibility feedback signals has led to the emergence of a rich literature. {In many cases, the literature focuses on specific classes of controllable loads, such as EVs~\cite{wenzel2017real}, heating, ventilation, and air conditioning (HVAC) systems~\cite{hao2014ancillary,wei2015proactive},  energy storage units~\cite{evans2019graphical}, thermostatically controlled loads~\cite{hao2014aggregate} or residential loads and pool pumps~\cite{sajjad2016definitions,meyn2015ancillary}.}  {In the context of these applications, a variety of approaches have been suggested, e.g., convex geometric approximations via virtual battery models~\cite{hao2014aggregate,zhao2017geometric}, hyper-rectangles~\cite{chen2018aggregating} and graphical interpretations~\cite{evans2019graphical};
scheduling based
aggregation~\cite{subramanian2013real,papadaskalopoulos2013decentralized}; linear combination of demand bit curves~\cite{wei2015proactive}; and probability-based characterization~\cite{sajjad2016definitions,meyn2015ancillary}.} These approaches have all yielded some success, especially in terms of quantifying available aggregate flexibility (see Section~\ref{sec:literature} for more detail on related work).  
However, nearly all prior work only focused on slower-timescale estimations and does not meet the goal of providing \emph{real-time} aggregate flexibility feedback. The fast-changing environment and the uncertainties of the DERs, however, demand real-time flexibility feedback. For example, in an EV charging facility, it is notoriously challenging to predict future EV arrivals and their battery capacities. With on-site solar generation, the aggregator's dynamical system can be time-varying and non-stationary, so it is crucial that real-time feedback be defined and approximated for it to be used in online feedback-based applications. {Furthermore, most of the existing frameworks are designated for specific tasks, such as managing HVAC systems~\cite{hao2014ancillary,wei2015proactive}, and therefore may not be applicable to other applications.}
Reinforcement learning (RL), especially, deep RL, has been used widely as approximation tools in smart grid applications. Joint pricing and EV charging scheduling for a single EV charger is considered in~\cite{wang2019reinforcement} using state–action–reward–state–action (SARSA). But it is unclear how the proposed method in~\cite{wang2019reinforcement} can be extended to allow multiple chargers.  Q-learning is used to estimate the residual energy in an energy storage system at the end of each day in~\cite{wang2015near} and determine the aggregate action
for thermostatically controlled loads (TCLs)~\cite{claessens2018model}. The authors in~\cite{chen2020learning} combine evolution strategies and model predictive control (MPC) to coordinate heterogeneous TCLs. Most existing studies, including the aforementioned works typically use RL for a ``central controller'' (which is an operator in our context).  Instead we use it for the aggregator to learn flexibility representations.

{
To the best of our knowledge, no paper has focused on the design of real-time coordination between an aggregator and a system operator that achieves the goals laid out above, except for some preliminary results in~\cite{li2020real,li2021information}.
Those results rely on a novel design of a real-time feedback signal that can be used to quantify the aggregate flexibility and coordinate real-time control. 
In this paper, we extend the design of the feedback signal to a more general dynamic system with time-varying and non-stationary constraints, and we mainly focus on how to apply the real-time feedback to practical applications (e.g., EV charging) in power systems.
Towards this goal, we propose a reinforcement learning based approach to approximate this feedback and further incorporate the feedback into a penalized predictive control (PPC) scheme. On the theory side, we prove the optimality of the proposed PPC scheme, and through extensive numerical tests, we validate the superior empirical performance of PPC over classic benchmarks, such as MPC. 
}

\subsection{Contributions.}  
In summary, to complement previous research, this paper considers a closed-loop control model formed by a system operator (central controller) and an aggregator (local controller) and propose a novel design of \textit{real-time aggregate flexibility} feedback, called the \textit{maximum entropy feedback} (MEF) that
quantifies the flexibility available to an aggregator. Based on the definition of MEF, we design a reward function, which allows MEF to be efficiently learned by model-free RL algorithms. Our main contributions are:
\begin{enumerate}

\item We introduce a model of the real-time closed-loop control system formed by a system operator and an aggregator. This work is the first to close the loop and both define a concise measure of aggregate flexibility and show how it can be used by the system operator in an online manner to optimize system objectives while respecting the constraints of the aggregator's loads. 

\item
Within this model we define the ``optimal'' real-time flexibility feedback as the solution to an optimization problem that maximizes the entropy of the feedback vector. The use of entropy in this context is novel and to the best of our knowledge, this article is among the first to rigorously define a notion for \textit{real-time} aggregate flexibility with provable properties. {In particular we show that the exact MEF
allows the system operator to maintain feasibility and enhance flexibility.}

\item {Furthermore, we propose a novel combination of control and learning by integrating model predictive control (MPC) and the defined MEF.} Using the MEF as a penalty term, we introduce an algorithm called the \textit{penalized predictive control} (PPC), which only requires the system operator to receive the MEF at each time, \textit{without} knowing the states and dynamics of the aggregator. {We also prove that, under certain regularity conditions, the actions given by PPC are optimal.}

\item Finally, we demonstrate the efficacy of the proposed scheme using real EV charging data from Caltech's ACN-Data~\cite{lee2019acn}. Our experiments show that by sending simple action signals generated by the PPC, a system operator is able to coordinate with an EV charging aggregator to satisfy almost all EV charging demands, while only knowing the MEF learned by a model-free off-policy RL algorithm. The PPC is also showed to achieve lower cost than MPC, which in addition needs to have access to the complete state of the loads.
\end{enumerate}

\subsection{Related literature.} 
\label{sec:literature}
The growing importance of aggregators for the integration of controllable loads and the challenge of defining and quantifying the flexibility provided by aggregators has led to the emergence of a rich literature.  Broadly, this work can be separated into three approaches.  

\textit{Convex geometric approximation.} The idea of representing the set of aggregate loads as a virtual battery model dates back to~\cite{hao2014characterizing,hao2014aggregate}. In~\cite{zhao2017geometric}, flexibility of an aggregation of thermostatically controlled loads (TCLs) was defined as the Minkowski sum of individual polytopes, which is approximated by the homothets of a virtual battery model using linear programming. The recent paper~\cite{chen2018aggregating} takes a different approach and defines the aggregate flexibility as upper and lower bounds so that each trajectory to be tracked between the bounds is disaggregatable and thus feasible. However, convex geometric approaches cannot be extended to generate real-time flexibility signals because the approximated sets cannot be decomposed along the time axis. In~\cite{bernstein2016aggregation}, a belief function of setpoints is introduced for real-time control. However, feasibility can only be guaranteed when each setpoint is in the belief set and this may not be the case for systems with memory.

\textit{Scheduling algorithm-driven analysis.} Scheduling algorithms that enable the aggregation of loads have been studied in depth over the past decade. {The authors of~\cite{ma2011decentralized,gan2012optimal} introduced a decentralized algorithm with a real-time
implementation for EV charging to track a given load profile.}
The authors of
\cite{subramanian2012real} considered the feasibility of matching a given power trajectory and show that causal optimal policies do not exist. In this work, aggregate flexibility was implicitly considered as the set of all feasible power trajectories. Three heuristic causal scheduling policies were compared and the results were extended to aggregation of deferrable loads and
storage in~\cite{subramanian2013real}. Furthermore, decentralized participation
of flexible demand from heat pumps and EVs was
addressed in~\cite{papadaskalopoulos2013decentralized}.   Notably, the flexibility signals that have emerged from this literature generally are applicable only to specific policies and DERs.

\textit{Probability-based characterization.} There is much less work on probabilistic methods. The aggregate flexibility of residential loads was defined based on positive and negative
pattern variations by analyzing collective behaviour of aggregate users~\cite{sajjad2016definitions}. A randomized and decentralized control architecture for systems of deferrable loads was proposed in~\cite{meyn2015ancillary}, with a linear time-invariant system approximation of the derived aggregate nonlinear model. Flexibility in this work was defined as an estimate of the proportion of loads that are operating.  {Our work falls into this category, but differs from previous papers in that entropy maximization for a closed-loop control system yields an interpretable signal that can be informative for operator objectives in real-time, as well as guarantee feasibility of the private constraints of loads (if the signal is accurate).} In our previous work~\cite{li2021information}, we study the problem of real-time coordination of an aggregator and a system operator under the paradigm of a control framework and provide regret analysis assuming feasibility predictions are available.

\textit{Other approaches.} Beyond the works described above, there are many other suggestions for metrics of aggregate flexibility, e.g., graphical-based measures~\cite{kara2015estimating} and data-driven approaches~\cite{kara2015estimating}. 
Most of these, and the approaches described above, are evaluated on the aggregator side only, and much less attention has been paid to the question of real-time coordination between an ISO and an aggregator that controls decentralized loads. 

The assessment and enhancement of aggregate flexibility are often considered independent of the operational objectives. For instance, in a reserve market, an aggregator will report to the ISO a day in advance an offline notion of aggregated flexibility based on forecast for the ISO to compute a energy and reserve schedule for the following day, e.g.,~\cite{hao2014characterizing,chen2019aggregate,chen2018aggregating,madjidian2018energy}, with notable exceptions, such as~\cite{wenzel2017real}, which considered charging and discharging of EV fleets batteries for tracking a sequence of automatic generation control (AGC) signals. However, this approach has several limitations. First, in large-scale systems, knowing the exact states of each load is not realistic. Second, classical flexibility representations often rely on a precise state-transition model on the aggregator's side. Third, traditional ISO market designs, such as a day-ahead energy market, often make use of ex ante estimates of future system states. The forecasts of the future states can sometime be far from reality, because of either an inaccurate model is used, or an uncertain event occurs. In contrast, a real-time energy market~\cite{marzband2013experimental,siano2016assessing} provides more robust system control when facing uncertainty in the  environment, e.g., from fast-changing renewable resources or human behavioral parameters. This further highlights the need for real-time flexibility feedback, and serves to differentiate the approach in our paper. {Below we present the notation frequently used in the remainder of this paper.}

\section*{Nomenclature Table}
	\addcontentsline{toc}{section}{Nomenclature}
	\subsection{System Operator (\textit{Central Controller})}
	\begin{IEEEdescription}[\IEEEusemathlabelsep\IEEEsetlabelwidth{$\mathsf{U}_t$}]
	\item[$\nt$] Total number of time slots.
	\item[$t$] Time index.
	\item[$u_t$] Operator action.
	\item[$c_t$] Cost function.
	\item[$C_{\nt}$] Cumulative costs.
	\item[$\psi_t$] Operator function.
	\item[$\beta_t$] Tuning parameter.
	\end{IEEEdescription}

	\subsection{Aggregator (\textit{Local Controller})}
	\begin{IEEEdescription}[\IEEEusemathlabelsep\IEEEsetlabelwidth{$\mathsf{U}_t$}]
		\item[$x_t$] Aggregator state.
		\item[$p_t$] Real-time aggregate flexibility feedback.
		\item[$\mathsf{X}_t$] Set of feasible states.
		\item[$\mathsf{U}_t$] Set of feasible actions.
		\item[$\mathsf{S}$] Set of feasible  action trajectories.
		\item[$f_t$] State transition function.
        \item[$\mathsf{P}$] Set of flexibility feedback.
        \item[$\phi_t$] Aggregator function.
	\end{IEEEdescription}
	\subsection{EV Charging Example}
	\begin{IEEEdescription}[\IEEEusemathlabelsep\IEEEsetlabelwidth{$u_t(j)$}]
	\item[$\nn$] Total number of accepted charging sessions.
	\item [$j$] Index of charging sessions.
	\item[$u_t$] Aggregate substation power level.
	\item[$s_t$] Charging decision vector.
	\item[$s_t(j)$] Scheduled energy.
	\item[$a(j)$] Arrival time.
	\item[$d(j)$] Departure time.
	\item[$e(j)$] Total energy to be delivered.
	\item[$r(j)$] Peak charging rate.
	\item[$d_t(j)$] Remaining charging time.
	\item[$e_t(j)$] Remaining energy demand.
	\item[$\Delta$] Time unit.
	\end{IEEEdescription}

\textbf{Notation and Conventions. } 
We use $\mathbbm{P}\left(\cdot\right)$ and $\mathbbm{E}\left(\cdot \right)$ to denote the probability distribution and expectation of random variables. The (differential) entropy function is denoted by $\mathbbm{H}(\cdot)$. To distinguish random variables and their realizations, we follow the convention to denote the former by capital letters (e.g., $U$) and the latter by lower case letters (e.g., $u$). Furthermore, we denote the length-$t$ prefix of a vector $u$ by $u_{\leq t}:=(u_1,\ldots,u_{t})$. Similarly, $u_{<t}:=(u_1,\ldots,u_{t-1})$ and $u_{a\rightarrow b}:=(u_{a},\ldots,u_{b})$. The concatenation of two vectors $u$ and $v$ is denoted by $(u, v)$. Given two vectors $u,v\in\mathbbm{R}^{\nn}$, we write $ u\preceq v$ if $u_i\leq v_i$ for all $i=1,\ldots,\nn$. For $x\in\mathbbm{R}$, denote $[x]_+:=\max\{0,x\}$. {The set of non-negative real numbers is denoted by $\mathbb{R}_+$.} 

The rest of the paper is organized as follows. We present our closed-loop control model in Section~\ref{sec:model}. We define real-time aggregate flexibility, called the MEF, and prove its properties in Section~\ref{sec:feedback}. An RL-based approach for estimating the MEF is provided in Section~\ref{sec:learning}. Combining MEF and model MPC, we propose an algorithm, termed the PPC in Section~\ref{sec:PPC}. Numerical results are given in Section~\ref{sec:experiment}. Finally, we conclude this paper in Section~\ref{sec:conclusion}.

\section{Problem Formulation}
\label{sec:model}



In this paper, we consider a real-time control problem involving two parties -- a \textit{load aggregator} and an independent system operator (\textit{ISO}), or simply called an \textit{operator} that interact over a discrete time horizon $[T] := \{1, \dots, T\}$. 

\subsection{Load aggregator}

A \textit{load aggregator} is a device, often considered as a local controller that controls a fleet of controllable loads. In this part, we formally state the model of an aggregator and its objective.
Let $x_t$ denote the \emph{aggregator state} at time $t$ that takes value in a certain set $\mathsf{X}\subseteq\mathbb{R}^{\nm}$.
To this end, the aggregator receives an \textit{action} $u_t\in\mathsf{U}$ where {$\mathsf{U}\subseteq\mathbb{R}$ denotes a closed and bounded set of actions at each time $t$ from a \textit{system operator}}, which will be formally defined in Section~\ref{operator}.  
The action space $\mathsf{U}$ and state space $\mathsf{X}$ are prefixed and known as common knowledge to both the aggregator and the system operator. The goal of the aggregator is to accomplish a certain task over the horizon $[\nt]$, e.g., delivering energy to a set of EVs by their deadlines while minimizing the costs, subject to system constraints. Mathematically, the constraints are represented by two collections of \textit{time-varying} and \textit{time-coupling} sets $\{\mathsf{X}_t(x_{<t},u_{<t})\subseteq\mathsf{X}:t\in [\nt]\}$ and $\{\mathsf{U}_t(x_{<t},u_{<t})_t\subseteq\mathsf{U}:t\in [\nt]\}$. For notational simplicity, we denote  $\mathsf{X}_t(x_{<t},u_{<t})$ by $\mathsf{X}_t$ and $\mathsf{U}_t(x_{<t},u_{<t})$ by $\mathsf{U}_t$ in the remaining contexts. The states and actions must satisfy $x_t\in\mathsf{X}_t$ and $u_t\in\mathsf{U}_t$ for all $t\in [\nt]$.
The decision changes the aggregator state $x_t$ according to a \textit{state transition function} $f_t$:
\begin{align}
    \label{eq:system}
    x_{t+1} = f_t(x_t,u_t),
    \ \ &x_{t} \in \mathsf{X}_t, \ 
    u_{t}\in\mathsf{U}_t,\end{align}
where $f_t$ represents the transition of the state $x_t$. The initial state $x_1$ is assumed to be the origin without loss of generality. The aggregator state $x_t$ and decision $u_t$ need to be chosen from two time-varying sets $\mathsf{X}_t$ and $\mathsf{U}_t$.
{ We make the following model assumptions:
\begin{assumption}
\label{ass:1}
{The dynamic $f_t(\cdot,\cdot):\mathsf{X}_{t}\times\mathsf{U}_{t}\rightarrow\mathsf{X}_{t+1}$ is a Borel measurable function for $t\in [\nt]$.} {The time-varying and time-coupling sets $\{\mathsf{U}_t:t\in [\nt]\}$ and  $\{\mathsf{X}_t:t\in [\nt]\}$ are Borel sets in $\mathbb{R}$ and $\mathbb{R}^{\nm}$.}
\end{assumption}}

The aggregator has flexibility in its actions $u_t$ for accomplishing its task and, we assume for this paper, is indifferent to these decisions as long as the task is accomplished by time $T$.
{At each time $t$, based on its current state $x_t$, the aggregator needs to send \textit{flexibility feedback}, $p_t$, a probability density function, from a collection of feedback signals $\mathsf{P}$, to the system operator, which describes the flexibility of the aggregator for accepting different actions $u_t$. We formally define $p_t$ and $\mathsf{P}$ in Section~\ref{sec:def_of_feedback}.} Designing $p_t$ is one of the central problems considered in this paper { (see Section~\ref{sec:feedback} for more details)}. Below we state the aggregator's goal in the real-time control system.

\begin{aggregator}
The goal of the aggregator is two-fold: (1). {Maintain the feasibility of the system} and guarantee that $x_t\in\mathsf{X}_t$ and $u_t\in\mathsf{U}_t$ for all $t\in [\nt]$. (2). Generate flexibility feedback $p_t$ and send it to the operator at time $t\in [\nt]$.
\end{aggregator}

\begin{remark}
\label{remark:action_space}
\emph{{We assume that the action space $\mathsf{U}$ is a continuous set in $\mathbb{R}$ only for simplicity of presentation. The results and definitions in the paper can be extended to discrete setting by changing the integrals to summations, and replacing the differential entropy functions by discrete entropy functions, e.g., see the definition of maximum entropy feedback (Definition~\ref{def:MEF}) and Lemma~\ref{lemma:explicit}. In practical systems e.g., an electric system consisting of an EV aggregator and an operator, $\mathsf{U}$ often represents the set of power levels and when the gap between power levels is small, $\mathsf{U}$ can be modeled as a continuous set.}}
\end{remark}

\begin{figure*}[htbp]
	\centering
	\includegraphics[scale=0.25]{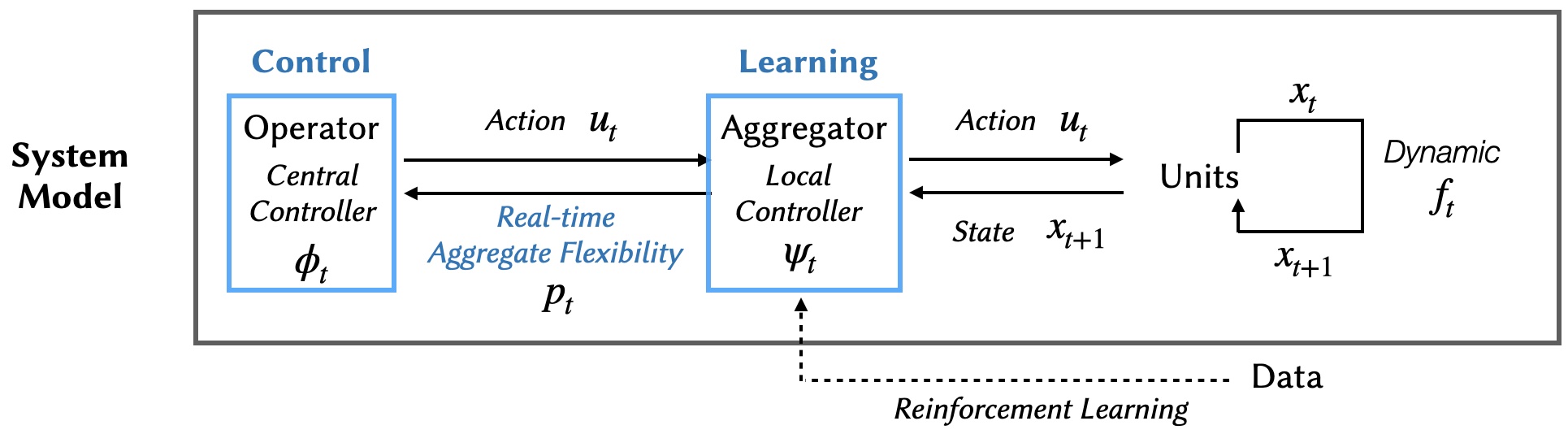}
	\caption{{System model: A feedback control approach for solving an online 
	version of \eqref{eq:offline}. The operator implements a control algorithm and the aggregator uses reinforcement learning to generate real-time aggregate flexibility feedback.}}
	\label{fig:framework}
	\medskip
\end{figure*}

\subsection{System operator}
\label{operator}

A system operator is a central controller that operates the power grid.
Knowing the flexibility feedback $p_t$ from the aggregator, the operator sends an action $u_t$, chosen from $\mathsf{U}$ to the aggregator at each time $t\in [\nt]$. Each action is associated with a cost function $c_t(\cdot):\mathsf{U}\rightarrow\mathbbm{R}_+$, e.g., the aggregate EV charging rate increases
load on the electricity grid. The system's objective is stated as follows.



\begin{operator}
The goal of the system operator is to provide an action $u_t\in\mathsf{U}$ at time $t\in [\nt]$ to the aggregator so as to minimize the cumulative system costs given by $C_\nt(u_1,\ldots,u_{\nt}) :=\sum_{t=1}^{\nt}c_t(u_t)$.
\end{operator}


\subsection{Real-time operator-aggregator coordination}
\label{sec:coordination}
Overall, considering the aggregator and operator's objectives, the goal of the  closed-loop system is to solve the following problem in \textit{real-time}, by coordinating the operator and aggregator via $\{p_t:t\in [\nt]\}$ and $\{u_t:t\in [\nt]\}$:
\begin{subequations}
\begin{align}
\label{eq:offline_1}
\min_{u_1,\ldots,u_{\nt}} & C_\nt(u_1,\ldots,u_{\nt})\\
\nonumber
\text{subject to } \forall t&=1,\ldots, \nt:\\
\label{eq:offline_2}
x_{t+1} &= f_t(x_t,u_t)\\
x_t &\in \mathsf{X}_t,\\
\label{eq:offline_3}
u_t &\in\mathsf{U}_t
\end{align}
i.e., 
\label{eq:offline}
\end{subequations}
the operator aims to minimize its cost $C_{\nt}$ in~\eqref{eq:offline_1} while the load aggregator 
needs to fulfill its obligations in the form of constraints 
\eqref{eq:offline_2}-\eqref{eq:offline_3}.  
This is an offline problem that involves global information at all
times $t\in [\nt]$. 

\begin{remark}
\label{remark:online}
\emph{For simplicity, we describe our model in an offline setting where 
the cost and the constraints in the optimization
problem \eqref{eq:offline} 
are expressed in terms of the entire trajectories of states and actions. The goal of the closed-loop control system is, however, to solve an online optimization via operator-aggregator coordination.} 
\end{remark}

The challenges are: (i) the aggregator and operator need to solve the online version of~\eqref{eq:offline} jointly, and (ii) the cost function $C_{\nt}$ is private to the operator and the constraints \eqref{eq:offline_2}-\eqref{eq:offline_3} are private to the operator.  
It is impractical for the aggregator to
communicate the constraints
to the operator because of privacy concerns or computational effort. Moreover, in an online setting, even the aggregator will not know the  constraints that involve future information, e.g., future EV arrivals in an EV charging station. { Formally, at each time $t\in [\nt]$, we assume that the operator and aggregator have access to the following information respectively:
\begin{enumerate}
    \item \textit{An operator knows the costs $(c_1,\ldots,c_t)$ and feedback $(p_1,\ldots,p_t)$, but not the future costs  $(c_{t+1},\ldots,c_{\nt})$ and feedback $(p_{t+1},\ldots,p_{\nt})$.}
    \item \textit{An aggregator knows the state transition functions $(f_1,\ldots,f_{\nt})$, { the initial state $x_1$ }and actions $(u_1,\ldots,u_{t})$.}
\end{enumerate}

{
\begin{system}
Overall, the goal of a aggregator-operator system is to jointly solve the online version of~\eqref{eq:offline_1}-\eqref{eq:offline_3} whose partial information is known to an aggregator and an operator respectively.
\end{system}}

{}

\subsection{Necessities of combining learning and control}

With the assumptions above, on the one hand the aggregator cannot solve the problem independently because it does not have cost information (since the costs are often sensitive and only of the operator's interests) from the operator and even if the aggregator could, it may not have enough power to solve an optimization to obtain an action.  On the other hand, the operator has to receive flexibility information from the aggregator in order to act. Well-known methods in pure learning or control cannot be used for this problem directly. From a learning perspective, the aggregator cannot simply use reinforcement learning and transmit parameters of a learned Q-function or an actor-critic model to the operator because the aggregator does not know the costs. From a control perspective, although model predictive control (MPC) is widely used for EV charging scheduling in practical charging systems ~\cite{lee2018large, lee2019acn}, it requires precise state information of electric vehicle supply equipment (EVSE). Thus, to solve the induced MPC problem, the system operator or aggregator needs to solve an online optimization at each time step that involves hundreds or even thousands of variables. This not just a complex problem, but the state information of the controllable units is potentially sensitive. This combination makes controlling sub-systems using precise information impractical for a future smart grid~\cite{li2020real,li2021information}
In this work, we explore a solution where the system operator and the 
aggregator jointly solve an online version of \eqref{eq:offline} in
a closed loop, as illustrated in Figure~\ref{fig:framework}.

The real-time operator-aggregator coordination illustrated in Figure~\ref{fig:framework} combines learning and control approaches. It does not require the aggregator to know the system
operator's objective in \eqref{eq:offline_1}, but only the
action $u_t$ at each time {$t\in [\nt]$} from the operator.
In addition, it does not require the system operator to know the aggregator 
constraints in \eqref{eq:offline_2}, but only a feedback signal 
$p_t$ (to be designed) from the aggregator. After receiving flexibility feedback $p_t$, which could be generated by machine learning algorithms, the system
operator outputs an action $u_t$ using a causal \textit{operator function} $\phi_t(\cdot):\mathsf{P}\rightarrow\mathsf{U}$. Knowing the state $x_{t}$,
the aggregator generates its feedback $p_t$ using a
causal \textit{aggregator function}  $\psi_t(\cdot):\mathsf{X}\rightarrow\mathsf{P}$ where $\mathsf{P}$ denotes the domain of flexibility feedback that will be formally defined in Section~\ref{sec:def_of_feedback}. By an ``online feedback'' solution, we mean
that these functions $(\phi_t, \psi_t)$ use {only information available 
locally at time $t\in [\nt]$}.

In summary, the closed-loop control system in our model proceeds as follows.  At each time $t$, the aggregator learns or computes a length-$|\mathsf{U}|$ vector $p_t$
based on previously received action
trajectory ${u}_{<t}=(u_1,\ldots,u_{t-1})$, and sends it to
the system operator.\footnote{We will omit ${u}_{<t}$ in the notation when it is not 
essential to our discussion and simplify the probability vector as $p_t$. Note that in~\eqref{eq:af3} we slightly abuse the notation and use $p_t$ to denote a conditional distribution. This is only for computational purposes and the information sent from an aggregator to an operator at time $t\in [\nt]$ is still a length-$|\mathsf{U}|$ probability vector, conditioned on a fixed $u_{<t}$.} 
The system operator thencomputes a (possibly random) action
$u_t  =  \phi_t(p_t)$
based on the flexibility feedback $p_t$ and sends it to the aggregator. 
The operator chooses its signal $u_t$ in order to solve the time-$t$ problem
in an online version of \eqref{eq:offline}, so the function $\phi_t$ denotes
the mapping from the flexibility feedback $p_t$ to an optimal solution
of the time-$t$ problem. See~\ref{sec:PPC} for examples.
The aggregator then computes the next feedback $p_{t+1}$ and the cycle repeats; see Algorithm~\ref{alg:online_control}. The goal of this paper is to provide concrete constructions of an aggregator function $\psi$ (as an MEF generator; see Section~\ref{sec:feedback}) and an operator function $\phi$ (via the PPC scheme; see Section~\ref{sec:PPC}).

\begin{algorithm}[t!]
\small
	\hrule
	\hrule
	\medskip
 		\For{$t\in [\nt]$}{
 				
	\textit{Operator {(Central Controller)}}	
  		
  	\quad	Generate actions using the PPC:
  		\begin{align*}
  		{u}_{t} &= \phi_t\left(p_t\right)\\
  		C_t &= C_{t-1} + c_t(u_t)
  		\end{align*}
  				
      	\textit{Aggregator {(Local Controller)}}
       	
      \quad 	Update system state:
      	\begin{align*}
      	    x_{t+1} = f_t(x_t,u_t)
      	\end{align*}
       	
      \quad  Compute estimated MEF:
         \begin{align*}
             p_{t+1} = \psi_t(x_{t+1})
         \end{align*}
        }
    
{Return {Total cost $C_{\nt}$}\;}
	\hrule
	\hrule
	\medskip
\caption{Closed-loop online control framework of a system operator (central controller) and an aggregator (local controller).}
	\label{alg:online_control}
\end{algorithm}

In the sequel, we demonstrate our system model using an EV charging application, as an example of the problem stated in~\eqref{eq:offline}.

\subsection{{An EV Charging Example}}
\label{sec:example}

Consider an aggregator that is an EV charging facility with $\nn$ accepted users. 
Each user $j$ has a private vector $\left(a(j), d(j),e(j),r(j)\right)\in\mathbbm{R}^4$ where $a(j)$ denotes its arrival (connecting) time; $d(j)$ denotes its departure (disconnecting) time, normalized according to the time indices in $[\nt]$; $e(j)$ denotes the total energy to be delivered, and $r(j)$ is its peak charging rate. 
Fix a set of $\nn$ users with their private vectors $\left(a(j), d(j),e(j),r(j)\right)$, 
the \textit{aggregator state} $x_t$ at time $t\in [\nt]$ is a collection of
length-$2$ vectors $(d_t(j), e_t(j) : a(j) \leq t \leq d(j))$ for each EV that
has arrived and has not departed by time $t$.  
Here $e_t(j)$ is the remaining energy 
demand of user $j$ at time $t$ and $d_t(j)$ is the remaining charging time.
The decision $s_t(j)$ is the energy delivered to each user $j$ at time $t$, determined by a \textit{scheduling policy} $\pi_t$ such as the well-known earliest-deadline-first, least-laxity-first, etc. Let $s_t:=(s_t(1),\ldots,s_t(\nn))$ and we have $s_t = \pi_t(u_t)$ where $u_t$ in this example is the aggregate substation power level, chosen from a small discrete set $\mathsf{U}$.
The aggregator decision $s_t(j) \in{\mathbbm{R}}_{+}$ at each time $t$
updates the state, in particular $e_t(j)$ such that
\begin{subequations}
\begin{align}
\label{eq:ev_transition_a}
    e_t(j) &= e_{t-1}(j) - s_t(j)\\
    \label{eq:ev_transition_b}
      d_t(j) & = d_{t-1}(j)  - \Delta
\end{align}
\end{subequations}
where $\Delta$ denotes the time unit and we assume that there is no energy loss. The laws~\eqref{eq:ev_transition_a}-\eqref{eq:ev_transition_b} are examples of the generic transition functions $f_1,\ldots,f_{\nt}$ in~\eqref{eq:system}.

Suppose, in the context of demand response, the system operator 
(a local utility company, or a building management) sends
a signal $u_t$ that is the aggregate energy that can be allocated to EV charging.
The aggregator makes charging decisions $s_t(j)$ to track the signal $u_t$ received from the system operator as long as they will meet the energy demands 
of all users before their deadlines. 
Then the constraints in~\eqref{eq:offline_2}-\eqref{eq:offline_3} are the following constraints on 
the charging decisions $s_t$, as a function of $u_t$:
\begin{subequations}
\begin{align}
\label{eq:f1}
s_{t}(j)= 0 \ , \  t<a(j), & \ j=1,\ldots,\nn,\\
\label{eq:f5}
s_{t}(j)= 0\ , \ t>d(j), & \  j=1,\ldots,\nn,\\
\label{eq:f3}
\sum_{j=1}^{\nn}s_{t}(j)= u_{t}, & \ t=1,\ldots,\nt,\\
\label{eq:f2}
\sum_{t=1}^{\nt}s_{t}(j) = e(j), & \ j=1,\ldots,\nn,\\
\label{eq:f4}
0\leq s_{t}(j)\leq r(j), & \ t=1,\ldots,\nt
\end{align}
\label{eq:f123}
\end{subequations}
{In above, constraint \eqref{eq:f3} ensures that the aggregator decision $s_t$ tracks the signal $u_{t}$ at each time $t\in [\nt]$,
the constraint (\ref{eq:f2}) guarantees that EV $j$'s energy demand is satisfied,  and the other constraints say that the aggregator cannot charge an EV before its arrival, after its departure, or at a rate that exceeds its limit. Inequalities~\eqref{eq:f1}-\eqref{eq:f4} above are examples of the constraints in~\eqref{eq:system}. Together, for this EV charging application,~\eqref{eq:ev_transition_a}-\eqref{eq:ev_transition_b} and~\eqref{eq:f1}-\eqref{eq:f4} exemplify the dynamic system in~\eqref{eq:system}.}

{The system operator's objective to minimize the cumulative costs $C_{\nt}(u):=\sum_{t=1}^{\nt}c_t u_t$ where $u=(u_1,\ldots,u_{\nt})$ are substation power levels, as outlined in Section~\ref{operator}. The cost $c_t$ depends on multiple factors such as the electricity prices and injections from an installed rooftop solar panel. Overall, the EV charging problem is formulated below, as a specific example of the generic optimization~\eqref{eq:offline_1}-\eqref{eq:offline_3}: 
\begin{subequations}
\begin{align}
\label{eq:ev_1}
\min_{u_1,\ldots,u_{\nt}} & \sum_{t=1}^{\nt}c_t u_t\\
\label{eq:ev_2}
\eqref{eq:ev_transition_a}-\eqref{eq:ev_transition_b} \ &\text{and} \ \eqref{eq:f1}-\eqref{eq:f4}
\end{align}
\label{eq:ev}
\end{subequations}

}

\section{Definitions of Real-time Aggregate Flexibility: Maximum Entropy Feedback}
\label{sec:feedback}

In this section we propose a specific function $\psi_t$
in the class defined by \eqref{eq:psit.1} for computing
flexibility feedback to quantify its future flexibility.  
We will justify our proposal by showing that the proposed $\psi_t$ 
has several desirable properties for solving an online 
version of~\eqref{eq:offline} using the real-time feedback-based approach described in Section~\ref{sec:model}.

\subsection{Definition of Flexibility Feedback $p_t$}
\label{sec:def_of_feedback}

A major challenge in our problem is that the operator has access to neither the feasible set nor the dynamics directly. Therefore, a notion termed  \textit{aggregate flexibility} has to be designed. It is often a ``simplified'' summary of the constraints in~\eqref{eq:offline_2}-\eqref{eq:offline_3}, as we reviewed in Section~\ref{sec:literature}. Notably, existing aggregate flexibility definitions (for instance, in~\cite{hao2014characterizing,hao2014aggregate,sajjad2016definitions,zhao2017geometric,madjidian2018energy,chen2018aggregating,sadeghianpourhamami2018quantitive,evans2019graphical}) all focus on the offline version of~\eqref{eq:offline}. It remains unclear that first, \textit{what is the right notion of \textit{real-time aggregate flexibility}? i.e., what is the right form of the flexibility feedback $p_t$?} Second, \textit{how can this $p_t$  be used by an operator?}

{In the following, we present a design of the flexibility feedback $p_t$, which is first proposed in our previous work~\cite{li2020real} for discrete $\mathsf{U}$ and~\cite{li2021information} for continuous $\mathsf{U}$. It
quantifies future flexibility that will be enabled by an operator 
action $u_t$.} The feedback $p_t$ therefore is a surrogate for the
aggregator constraints \eqref{eq:offline_2} to guide the operator's decision. Let $u:=(u_1,\ldots,u_{\nt})$.
Specifically, define the set of all \textit{feasible action trajectories} for the aggregator 
as:
\begin{align*}
\mathsf{S}:=\left\{u\in\mathsf{U}^{\nt}: u\text{ satisfies } \eqref{eq:offline_2}-\eqref{eq:offline_3} \right\}.
\end{align*}

{
The following property of the set $\mathsf{S}$ is useful, whose proof can found in Appendix~\ref{app:measurability}.
\begin{lemma}
\label{lemma:measurability}
The set of feasible action trajectories $\mathsf{S}$ is Borel measurable.
\end{lemma}
}

Existing aggregate flexibility  definitions focus on approximating $\mathsf{S}$ such as finding its convex approximation (see Section~\ref{sec:literature} for more details). Our problem formulation needs a \textit{real-time} approximation of this set $\mathsf{S}$, i.e., decompose  $\mathsf{S}$ along the time axis $t=1,\ldots,\nt$.
Throughout, we assume that $\mathsf{S}$ is non-empty. 
Next, we define the space of flexibility feedback $p_t$.
Formally, we 
let $\mathsf P$ denote a set of density functions $p_t(\cdot):\mathsf{U}\rightarrow [0,1]$ that maps an action to a value in $[0,1]$ and satisfies
\begin{align*}
    \int_{u\in\mathsf{U}} p(u)\mathrm{d}u = 1.
\end{align*}

Fix $x_t$ at time $t\in [\nt]$. The aggregator function $\psi_t(\cdot): \mathsf{X} \rightarrow\mathsf{P}$ at each time $t$ outputs: 
\begin{align}
\label{eq:psit.1}
\psi_t(x_t) = p_t(\cdot|u_{<t})
\end{align}
such that $p_t(\cdot|u_{<t}):\mathsf{U}\rightarrow [0,1]$ is a conditional density function in $\mathsf{P}$.
We refer to $p_t$ as \emph{flexibility feedback} sent at time 
$t\in [\nt]$ from the aggregator to the system operator.
In this sense, \eqref{eq:psit.1} does not specify a specific
aggregator function $\psi_t$, but a class of possible functions
$\psi_t$.
Every function in this collection is \emph{causal} in that it
depends only on information available to the aggregator at time
$t$.
In contrast to most aggregate flexibility notions in the literature~\cite{hao2014characterizing,hao2014aggregate,sajjad2016definitions,zhao2017geometric,madjidian2018energy,chen2018aggregating,sadeghianpourhamami2018quantitive,evans2019graphical}, the 
flexibility feedback  here 
is specifically designed for an online feedback control setting.

\subsection{Maximum entropy feedback}

The intuition behind our proposal is using the conditional
probability $p_t(u_t|u_{<t})$ to measure 
the resulting future flexibility of the aggregator
if the system operator chooses $u_t$ as the signal at 
time $t$, given the action trajectory up to time $t-1$.
The sum of the conditional entropy of $p_t$ thus is a 
measure of how informative the overall feedback is.
This suggests choosing a conditional distribution $p_t$ that 
maximizes its conditional entropy.
Consider the optimization problem:
\begin{subequations}
\begin{align}
\label{eq:af1}
{\digamma} \ := \ \max_{p_1,\ldots,p_{\nt}}\ \sum_{t=1}^{\nt}\mathbbm{H}\left({U}_t|U_{<t}\right)\ 
    \text{subject to} \ U\in\mathsf{S}
\end{align}
where the variables are conditional density functions:
\begin{align}
p_t & \ := \ p_t(\cdot|\cdot):=\mathbbm{P}_{U_t|U_{<t}}(\cdot|\cdot),
\qquad t \in [\nt],
\label{eq:af2}
\end{align}
$U\in \mathsf U$ is a random variable distributed according to the joint distribution $\prod_{t=1}^{\nt}p_t$ and $\mathbbm{H}\left({U}_t|U_{<t}\right)$ is the differential conditional entropy of $p_t$ defined as:
\begin{align}
\nonumber
\mathbbm{H}\left({U}_t|U_{<t} \right):= 
\int_{u_{\leq t}\in\mathsf{U}^{t}}\Big(-\prod_{\ell=1}^{t}&p_\ell(u_\ell|u_{<\ell})\Big)\\
 \label{eq:af3}
& \ \log{p_t(u_{t}|u_{<t})}\mathrm{d}u_{\leq t}.
\end{align}
By definition, a quantity conditioned on ``$u_{<1}$'' means an unconditional 
quantity, so in the above, $\mathbbm{H}\left({U}_1|U_{<1} \right) := 
\mathbbm{H}\left({U}_1\right) := \mathbbm{H}\left({p}_1\right)$.

The chain rule shows that $\sum_{t=1}^{\nt}\mathbbm{H}\left({U}_t|U_{<t}\right) = \mathbbm{H}\left(U\right)$.
Hence \eqref{eq:af} can be interpreted as maximizing the entropy
$\mathbbm{H}\left(U\right)$ of a random trajectory $U$ sampled according 
to the joint distribution $\prod_{t=1}^{\nt}p_t$, conditioned on $U$ satisfying $U\in\mathsf{S}$, where the maximization is over the collection
of conditional distributions $(p_1, \dots, p_T)$.

\label{eq:af}
\end{subequations}
\begin{definition}[Maximum entropy feedback]
\label{def:MEF}
The flexibility feedback $p_t^* = \psi^*_t(u_{<t})$ for $t\in [\nt]$
is called the maximum entropy feedback (MEF) if
$(p_1^*, \dots, p_T^*)$ is the unique optimal solution of \eqref{eq:af}.
\end{definition}

\begin{remark}
\emph{Even though the optimization problem \eqref{eq:af} involves variables $p_t$
for the entire time horizon $[\nt]$, the individual variables
$p_t$ in \eqref{eq:af3} are conditional probabilities that depend
only on information available to the aggregator at times $t$.
Therefore the maximum entropy feedback $p^*_t$ in Definition \ref{def:MEF}
is indeed causal and in the class of $p^*_t$ defined in \eqref{eq:psit.1}.
The existence of $p^*_t$ is guaranteed by Lemma~\ref{lemma:explicit}
below, which also implies that $p^*_t$ is unique.}
\end{remark}

We demonstrate Definition \ref{def:MEF} using a toy example.
\begin{example}[Maximum entropy feedback $p^*$]
\label{example:toy}
Consider the following instance of the EV charging example in Section~\ref{sec:example}. Suppose the number of charging time slots is $\nt=3$ and there is one customer, whose private vector is $(1,3,1,1)$ and possible energy levels are $0$ (kWh) and $1$ (kWh), i.e., $\mathsf{U}\equiv\{0,1\}$. Since there is only one EV, the scheduling algorithm $u$ (disaggregation policy) assigns all power to this single EV. For this particular choices of $x$ and $u$, the set of feasible trajectories is $\mathsf{S}=\{(0,0,1),(0,1,0), (1,0,0)\}$, shown in Figure~\ref{fig:toy_example} with the corresponding optimal conditional distributions given by~\eqref{eq:af}. 
\begin{figure}[h]
    \centering
    \includegraphics[scale=0.35]{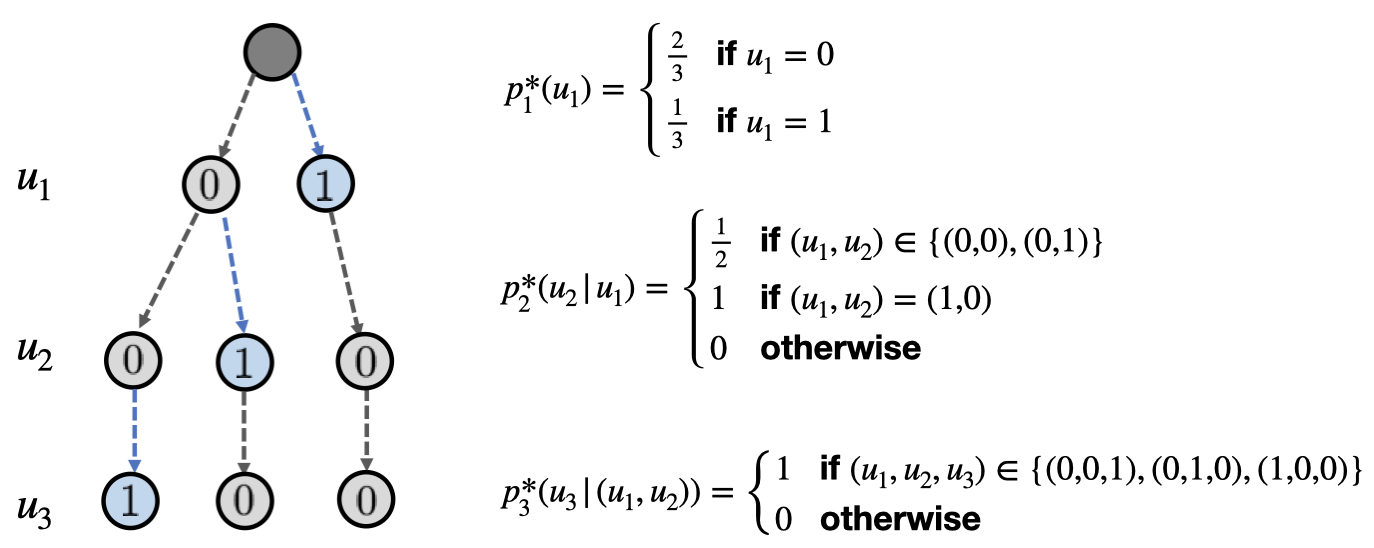}
    \caption{Feasible trajectories of power signals and the computed maximum entropy feedback in Example~\ref{example:toy}.}
    \label{fig:toy_example}
\end{figure}
\end{example}

\subsection{Properties of $p^*_t$}
\label{subsec:properties}

We now show that the proposed maximum entropy feedback $p^*_t$ has several desirable properties.
We start by computing $p^*_t$ explicitly. 
Given any action trajectory $u_{\leq t}$, define the 
set of \emph{subsequent} feasible trajectories as:
\begin{align*}
\mathsf{S}(u_{\leq t})
:=& \Big\{v_{>t}\in\mathsf{U}^{\nt-t}: v\text{ satisfies } \eqref{eq:offline_2}-\eqref{eq:offline_3}, v_{\leq t} = u_{\leq t} \Big\}.
\end{align*}
As a corollary 
The size 
$\left|\mathsf{S}(u_{\leq t})\right|$ of the set of subsequent
feasible trajectories is a measure of future flexibility, conditioned on
$u_{\leq t}$. Our first result justifies our calling $p^*_t$ the optimal
flexibility feedback: $p^*_t$ is a measure of the future flexibility that will be
enabled by the operator's action $u_t$ and it attains a measure of
system capacity for flexibility.
By definition, 
$p_1^*(u_1|u_{<1}) \ := \ p_1^*(u_1)$. 
{

\begin{lemma}
\label{lemma:explicit}
Let $\mu(\cdot)$ denote the Lebesgue measure. 
The MEF as optimal solutions of the maximization in (\ref{eq:af1})-(\ref{eq:af3}) are given by
\begin{align}
\label{eq:optimal}
p_t^*(u|u_{<t}) \equiv& \frac{\mu\left(\mathsf{S}((u_{<t}, u))\right)}{\mu\left(\mathsf{S}(u_{<t})\right)},
\quad\forall (u_{< t}, u_t) \in\mathsf{U}^{t}.
\end{align}
Moreover, the optimal value of~\eqref{eq:af1}-\eqref{eq:af3} is equal to $\log \mu(\mathsf{S})$.
\end{lemma}

\begin{remark}\emph{When the denominator $\mu\left(\mathsf{S}(u_{<t})\right)$ is zero, the numerator $\mu\left(\mathsf{S}((u_{<t}, u))\right)$ has also to be zero. For this case, we set $p_t^*(u|u_{<t})=0$ and this does not affect the optimality of~\eqref{eq:af1}-\eqref{eq:af2}.}
\end{remark}

The proof can be found in Appendix~\ref{app:proof_MEF}.}
{
The volume $\mu\left(\mathsf{S}\right)$ 
is a measure of flexibility inherent in the aggregator.  
We will hence call
$\log \mu\left(\mathsf{S}\right)$ the \emph{system capacity}.
Lemma~\ref{lemma:explicit} then says that the optimal value of~\eqref{eq:af}
is the system capacity, 
${\digamma}=\log \mu\left(\mathsf{S}\right)$.
Moreover
the maximum entropy feedback $(p^*_1, \dots, p^*_T)$ is the unique 
collection of conditional distributions that attains the system capacity
in~\eqref{eq:af}.
This is intuitive since the entropy of a random trajectory $x$ in $\mathsf{S}$ is maximized by the uniform distribution
$q^*$ in ~\eqref{eq:uniform} 
induced by the conditional distributions $(p_1^*,\ldots,p_{\nt}^*)$.}

Lemma~\ref{lemma:explicit} implies the following important properties of 
the maximum entropy feedback.
\begin{corollary}[Feasibility and flexibility]
\label{coro:property}
Let $p^*_t = p^*_t(\cdot|u_{<t})$ be the maximum entropy feedback  
at each time $t\in [\nt]$.
\begin{enumerate}
\item 
For any action trajectory $u=(u_1,\ldots,u_{\nt})$, if 
\begin{align*}
    p^*_t(u_t|u_{<t}) \ > \ 0 \quad \text{ for all } t\in [\nt]
\end{align*}
then $u\in\mathsf{S}$. 

\item 
For all $u_t,u_t'\in\mathsf{U}$ at each time $t\in [\nt]$, if
\begin{align*}
    p^*_t(u_t|u_{<t})\ \geq \ p^*_t(u_t'|u_{<t})
\end{align*}
then 
$\mu\left(|\mathsf{S}((u_{<t}, u_t))\right)\geq \mu\left(\mathsf{S}((u_{<t}  ,u_t'))\right)$.
\end{enumerate}
\end{corollary}

The proof is provided in Appendix~\ref{app:proof_feasibility}.
We elaborate the implication of Corollary \ref{coro:property}
for our online feedback-based solution approach.

\begin{remark}[Feasibility and flexibility]
\label{remark:feasibility}
\emph{Corollary \ref{coro:property} says that the proposed optimal flexibility
feedback $p^*_t$ provides the right information for the system operator to choose
its action $u_t$ at time $t$.
\begin{enumerate}
    \item
    (Feasibility)
    Specifically, the first statement of the corollary says that if the operator 
always chooses an action $u_t$ with positive conditional probability 
$p^*_t(u_t)>0$ for each time $t$, then the resulting action trajectory is 
guaranteed to be feasible, $u\in\mathsf{S}$, i.e., the system
will remain feasible at \emph{every} time $t\in [\nt]$ along the way.
\item
(Flexibility)
Moreover, according to the second statement of the corollary, if the system operator
chooses an action $u_t$ with a larger $p^*_t(u_t)$ value at time $t$, then 
the system will be more flexible going forward than if it had chosen another
signal $u_t'$ with a smaller $p^*_t(u_t')$ value, in the sense that there 
are more feasible trajectories in $\mathsf{S}((u_{<t}, u_t))$ 
going forward.
\end{enumerate}
}
\end{remark}
As noted in Remark \ref{remark:online}, despite characterizations 
that involve the whole action trajectory
$u$, such as $u\in\mathsf{S}$, these are \emph{online} properties.
This guarantees the feasibility of the online closed-loop control system 
depicted in Figure \ref{fig:framework}, and confirms the suitability 
of $p^*_t$ for online applications.

\begin{figure*}[htbp]
	\centering
	\includegraphics[scale=0.18]{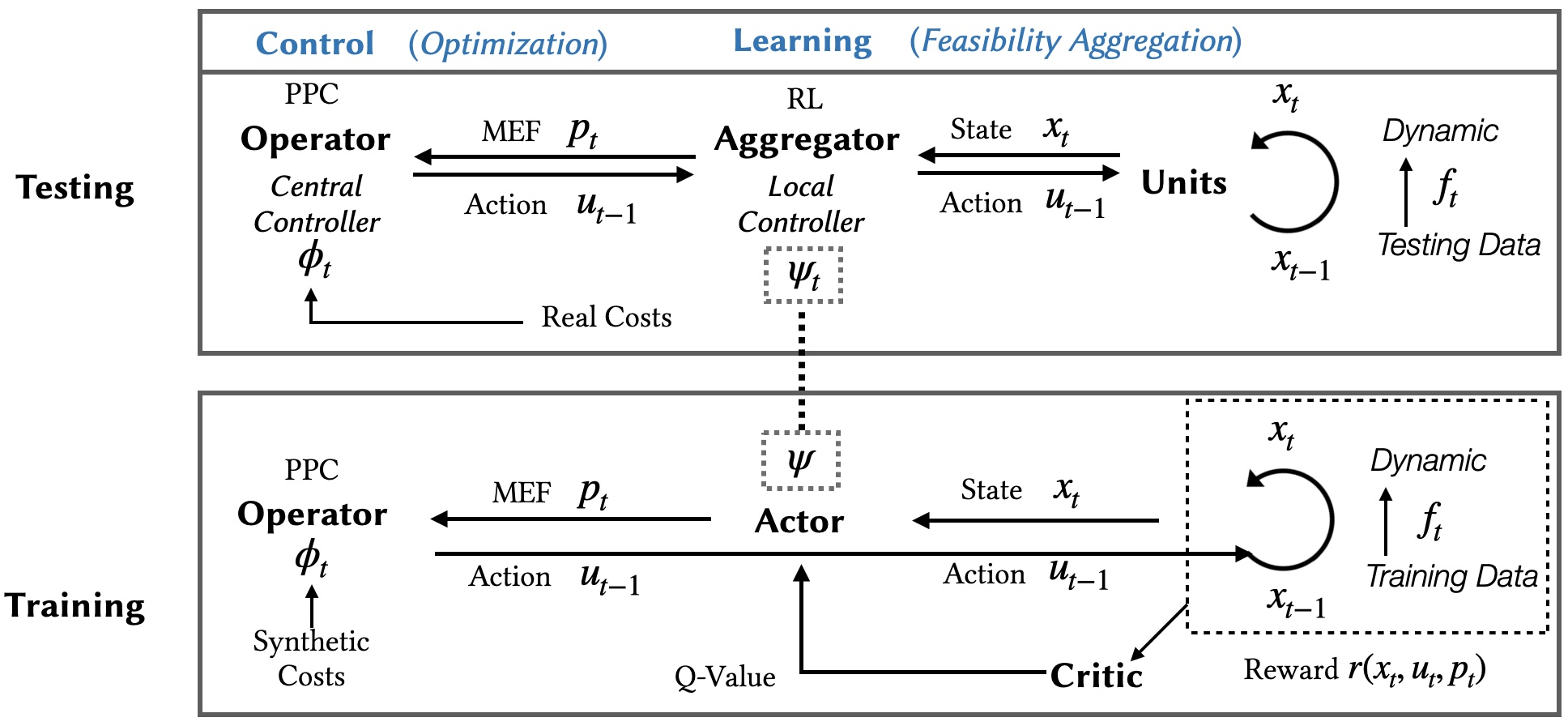}
	\caption{{Learning and testing architecture for learning aggregator functions.}}
	\label{fig:learning}
	\medskip
\end{figure*}

\section{Approximating Maximum Entropy Feedback via Reinforcement Learning}
\label{sec:learning}

For real-world applications, computing the maximum entropy feedback (MEF) could be computationally intensive. Thus, instead of computing it precisely, it is desirable to approximate it. {
In this section, we discuss the use of model-free reinforcement learning (RL) to generate an \textit{aggregator function} $\psi$. For practical implementation, we switch to the case when $\mathsf{U}$ is a discrete set and reuse the notation $\mathsf{P}$ to denote a probability simplex that contains all possible discrete MEF:
\begin{align}
\label{eq:simplex}
\mathsf{P}:=\left\{p\in\mathbbm{R}^{|\mathsf{U}|}:p(u)\geq 0, u\in\mathsf{U}; \sum_{u\in\mathsf{U}}p(u)=1\right\}.
\end{align}} 
We demonstrate that RL can be used to train a generator that outputs approximate MEF, given the state of the system. 
To be more precise, the learned aggregator function $\smash{{\psi}:\mathsf{X}\rightarrow \mathsf{P}}$ outputs an estimate of the MEF given the state $x_t$ at each time $t\in [\nt]$,
where $\mathsf{X}$ is the state space and $\mathsf{P}$ is the set of all possible MEF.
{Note that the aggregator does not know the cost functions, so it cannot directly use an RL algorithm and transmit the learned Q-function or actor-critic model to the operator. Moreover, even if the aggregator knows the cost functions, generating actions using RL needs to solve two contradicting tasks of both optimizing rewards and penalizing feasibility violations, which makes the design of reward function and reward clipping a challenging goal. In our approach, we separate the tasks of enforcing feasibility and minimizing costs. We generate MEF as feasibility signals via reinforcement learning methods, and optimize the operator's objective via a MPC-based method (introduced in Section~\ref{sec:PPC}). It is also worth noting that a number of effective heuristics may be available such as a greedy approximation in~\cite{li2020real} and other gradient-based or density estimation~\cite{singh2018nonparametric} methods.
We leave to future work the question of finding an optimal approximation algorithm.}

{
\subsection{Offline learning of aggregator functions}

To learn an aggregator function $\psi$ for estimating MEF, we use an actor-critic architecture~\cite{barto1983neuronlike} with separate policy and value function networks to enable the learning of policies on continuous action and state spaces. The actor-critic architecture is presented in Figure~\ref{fig:learning}, which shows the information update between actor and critic networks. Note that in practical actor-critic algorithms, typically the policy, Q-function(s) and value function(s) are modeled using deep neural networks and the parameters are updated using policy iteration via stochastic gradient descent. We omit those details in Figure~\ref{fig:learning}. 

\subsection{Training process}
During the training process, the data used for defining training dynamics are the episodes $(\mathsf{U}_t,\mathsf{X}_t,f_t)_{t=1}^{\nt}$. For example, for the EV charging application in Section~\ref{sec:example}, the training data of each episode (day) consist of historical private vectors $(a(j),d(j),e(j),r(j))$ specified by the users visited the charging station on the corresponding day.
Among actor-critic-based RL algorithms, off-policy actor-critic methods, such as deep deterministic policy gradient (DDPG)~\cite{lillicrap2015continuous} and soft actor-critic (SAC)~\cite{haarnoja2018soft} are known to attain better data efficiency in many applications. Below we take SAC, a maximum entropy deep RL algorithm, as an example to demonstrate the offline learning of an aggregator function $\psi$. In particular, for learning $\psi$, the objective of SAC is to maximize both the expected return and the expected entropy of the policy:
\begin{align}
\label{eq:sac_objective}
J(\psi) = \sum_{t=1}^{\nt} \mathbbm{E}_{(\overline{x}_t,p_t)\sim\rho_\psi}\left[r(\overline{x}_t,p_t)+\alpha \mathbbm{H}(\psi(\cdot|x_t)) \right]
\end{align}
where $\overline{x}_t:=(x_t,u_t)$; $\rho_\psi$ denotes the state-action marginals of the trajectory distribution induced by a policy $\psi$ and $r(\overline{x}_t,p_t)$ is a customized reward function. To estimate MEF, we need to determine a reward function $r(\overline{x}_t,p_t)$ in~\eqref{eq:sac_objective}. We adopt the following reward function that incorporates the constraints and the definition of MEF:
\begin{align}
\label{eq:reward}
r(\overline{x}_t,p_t) = &\mathbbm{H}(p_t) + \sigma g(\overline{x}_t; \mathsf{X}_t,\mathsf{U}_t)
\end{align}
where the first term is critical and it maximizes the entropy of the probability distribution $p_t$, based the definition of the MEF in Definition~\ref{def:MEF}; $g(\overline{x}_t)=g(x_t,u_t)$ is a function that rewards the state and action if they satisfy the constraints $x_t\in\mathsf{X}_t$ and $u_t\in\mathsf{U}_t$. The reward function is independent of the cost functions, which are synthetic costs in the training stage. A concrete example of $g(\overline{x}_t)$ is given in Section~\ref{sec:experiment}. We clip the output MEF given by the policy to make sure it is a probability vector in the probability defined in simplex~\eqref{eq:simplex}. In Figure~\ref{fig:reward}, a training curve is given and it displays the changes of rewards regarding to the number of training episodes. 

\subsection{Testing process}

With a trained aggregator function $\psi$ that tries to optimize $J(\psi)$ in~\eqref{eq:sac_objective}, we test the closed-loop system on new episodes defined by testing data, as shown in Figure~\ref{fig:learning}. The trained aggregator function (parameterized by a deep neural network) is used as a ``black box'' function that maps each state $x_{t}$ to feedback $p_t$.\footnote{In our model, in general the aggregator functions $\psi_1,\ldots,\psi_{\nt}$ can be time-dependent. In the offline learning process presented in this section, we use a single function to generate feedback.} Note that the real costs used in the testing process may not be same as the synthetic costs used in the training process, because the aggregator has no access to the costs as assumed in Section~\ref{sec:model}. 

In the sequel, with the learned MEF, we introduce a closed-loop framework that combines model predictive control (MPC) and RL to coordinate a system operator and an aggregator in real-time. { It is worth noting that the learned MEF may be different from the exact MEF provided in Definition~\ref{def:MEF}. However, later we show in Section~\ref{sec:experiment} that with the learned MEF, the constraints on the aggregator's side can almost be satisfied with a reasonable tuning parameter. In the EV charging example described in Section~\ref{sec:example}, this means the EV's batteries are fully charged; see Figure~\ref{fig:battery} for details.}

\begin{figure}[t]
	\centering
	\includegraphics[width=0.9\linewidth]{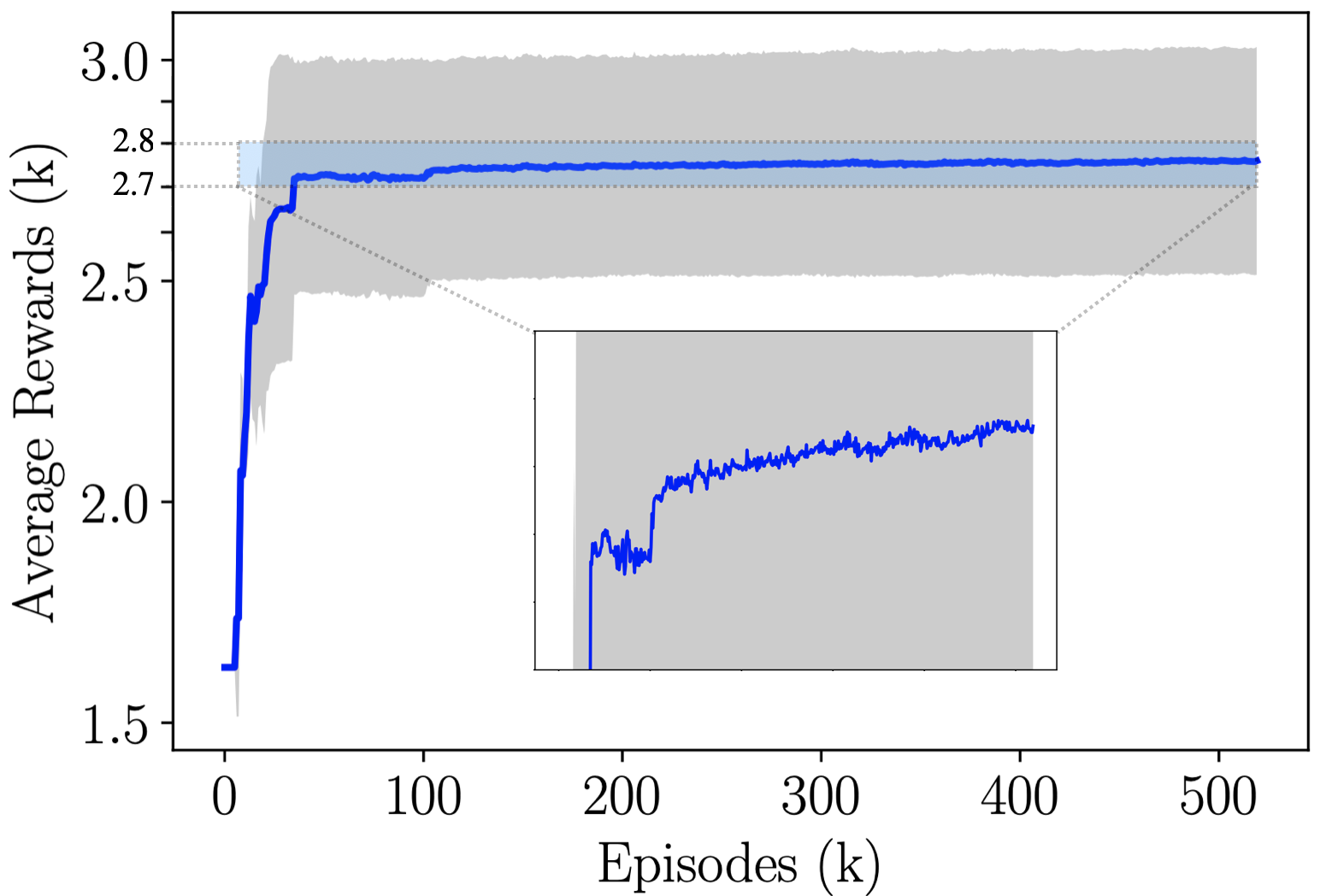}
	\caption{Average rewards (defined in~\eqref{eq:reward}) in the training stage with a tuning parameter $\beta=6\times 10^3$. Shadow region measures the variance.}
	\label{fig:reward}
	\medskip
\end{figure}
}

\section{Penalized Predictive Control}
\label{sec:PPC}

Consider the system model in Section~\ref{sec:model}. 
In this setting, the operator seeks to minimize the cost in an online manner, i.e., at time $t\in [\nt]$ the operator only knows the objective functions $c_1,\ldots,c_t$ and the flexibility feedback $p_1,\ldots,p_t$. The task of the operator is to, given the maximum entropy feedback, design a sequence of \textit{operator functions} $\phi_1,\ldots,\phi_{\nt}$ to generate actions $u_1,\ldots,u_{\nt}$ that are always feasible with respect to the constraints \emph{and} that minimize the cumulative cost.



\subsection{Key Idea: Maximum entropy feedback as a penalty term}

There is in general a trade-off between ensuring future flexibility and minimizing the current system cost in predictive control. The action $u_t$ guaranteeing the maximal future flexibility, i.e., having the largest $p^*_t(u_t|u_{<t})$ may not be the one that minimizes the current cost function $c_t$ and vice versa. Therefore, the online algorithm for the central controller must balance future flexibility and current cost. 
The key idea is to use MEF as a penalty term in the offline optimization problem.
Note that Corollary~\ref{coro:property} guarantees that the online agent can always find a feasible action $u\in\mathsf{S}$. Indeed, knowing the MEF $p^*_t$ for every $t\in [\nt]$ is equivalent to knowing the set of all admissible sequences of actions $\mathsf{S}$. {To see this, consider the unique maximum entropy feedback $(p^*_1,\ldots,p^*_{\nt})$ guaranteed
by Lemma~\ref{lemma:explicit} and let
$q(u)=\prod_{t=1}^{\nt}p^*_t(u_{t}|u_{<t})$ denote the joint distribution
of the action trajectory $u$.  Then \eqref{eq:optimal} implies that the joint
distribution $q$ is the uniform distribution over the set $\mathsf{S}$ of all feasible trajectories:
\begin{align}
\label{eq:uniform}
    q(u):=\begin{cases}
    1/\mu\left(\mathsf{S}\right) & \text{ if } u\in \mathsf{S}\\
    0 & \text{ otherwise}
    \end{cases}.
\end{align}}

{Using this observation, the constraints~\eqref{eq:offline_2}-\eqref{eq:offline_3} in the offline optimization can be rewritten as a penalty in the objective of~\eqref{eq:offline_1}. 
We present a useful lemma that both motivates our online control algorithm and builds up the optimality analysis in Section~\ref{sec:optimality}.}

\begin{lemma}
\label{lemma:connection}
The offline optimization~\eqref{eq:offline_1}-\eqref{eq:offline_3} is equivalent to the following unconstrained minimization for any $\beta>0$:
\begin{align}
\label{eq:unconstrained}
\inf_{u\in\mathsf{U}^{\nt}} \sum_{t=1}^{\nt} \left(c_t(u_t) - \beta\log p^*_t(u_t|{u}_{<t})\right)
\end{align}
\end{lemma}

The proof of Lemma~\ref{lemma:connection} can be found in Appendix~\ref{app:proof_connection}.
It draws a clear connection between MEF and the offline optimal, which we exploit in the design of an online system operator  in the next section. 

\subsection{Algorithm: Penalized Predictive Control via MEF}
\label{sec:PPC}

Our proposed design, termed penalized predictive control (PPC), is a combination of model predictive control (MPC) (\textit{c.f.}~\cite{camacho2013model}), which is a competitive policy for online optimization with predictions, and the idea of using MEF as a penalty term. This design makes a connection between the MEF and the well-known MPC scheme. The MEF as a feedback function, only contains limited information about the dynamical system in the local controller's side. (It contains only the feasibility information of the current and future time slots, as explained in Section~\ref{sec:feedback}). The PPC scheme therefore is itself a novel contribution since it shows that, even if \textit{only} feasibility information is available, it is still possible to incorporate the limited information to MPC as a \textit{penalty term}. 

We present PPC in Algorithm~\ref{alg:PPC}, where we use the following notation. 
{Let $\beta_t>0$ be a \textit{tuning parameter} in predictive control to trade-off the flexibility in the future and minimization of the current system cost at each time $t\in [\nt]$.} The next corollary follows whose proof is in Appendix~\ref{app:proof_feasibility}.
\begin{corollary}[Feasibility of PPC]
\label{corollary:feasibility}
{
When $p_t=p_t^*$ for all $t\in [\nt]$, the MEF defined in Definition~\ref{def:MEF}, the sequence of actions $u=(u_1,\ldots,u_{\nt})$ generated by the PPC in~\eqref{eq:PPC} always satisfies $u\in\mathsf{S}$ for any sequence of tuning parameters $(\beta_1,\ldots,\beta_{\nt})$. }
\end{corollary}

\begin{algorithm}[t!]
\small
	\hrule
	\hrule
	\medskip
	\KwData{Sequentially arrived cost functions and MEF}
	\KwResult{Actions $u=(u_1,\ldots,u_{\nt})$}
	
\For{$t = 1,\ldots,\nt$}{
  				Choose an action $u_t$ by minimizing:
\begin{align}
\label{eq:PPC}
        u_t = \phi_t(p_t):=
        &\arginf_{u_t\in \mathsf{U}} \left(c_t(u_t) - \beta_t\log p_t(u_t|u_{<
        t})\right)
        \end{align}
        }
        
{Return {$u$}\;}
	\hrule
	\hrule
	\medskip
\caption{Penalized Predictive Control (PPC).}
	\label{alg:PPC}
\end{algorithm}

\subsection{Framework: Closed-loop control between local and central controllers}
\label{sec:framework}
Given the PPC scheme described above,  we can now formally present our online control framework for the distant central controller and local controller (defined in Section~\ref{sec:model}). {Recall that an overview of the closed-loop control framework has been given in Algorithm~\ref{alg:online_control}, where $\phi$ denotes an operator function and $\psi$ is an aggregator function.}  To the best of our knowledge, this paper is the first to consider such a closed-loop control framework with limited information communicated in real-time between two geographically separate controllers seeking to solve an online control problem. We present the framework below.

At each time $t\in [\nt]$, the local controller first efficiently generates estimated MEF $p_t\in\mathsf{P}$ using an aggregator function $\psi_t$ trained by a reinforcement learning algorithm. After receiving the current MEF $p_t$ and cost function $c_t$ (future $w$ MEF and costs if predictions are available), the central controller uses the PPC scheme in Algorithm~\ref{alg:PPC} to generate an action $u_t\in\mathsf{U}$ and sends it back to the local controller. The local controller then updates its state $x_t\in\mathsf{X}$ to a new state $x_{t+1}$ based on the system dynamic in~\eqref{eq:system} and repeats this procedure again. In the next Section, we use an EV charging example to verify the efficacy of the proposed method.

{
\subsection{Optimality Analysis}
\label{sec:optimality}
To end our discussion of PPC we focus on optimality.  For the ease of analysis, we assume that the action space $\mathsf{U}$ is the set of real numbers $\mathbb{R}$; however, as noted in Remark~\ref{remark:action_space}, our system and the definition of MEF can also be made consistent with a discrete action space.


To understand the optimality of PPC we focus on standard regularity assumptions for the cost functions and the time-varying constraints. We assume cost functions are strictly convex and differentiable, which is common in practice. Further, let $\mu(\cdot)$ denote the Lebesgue measure. Note that the set of subsequent feasible action trajectories $\mathsf{S}(u_{\leq t})$ is Borel-measurable for all $t\in [\nt]$, implied by the proof of Corollary~\ref{lemma:measurability}. We also assume that the measure of the set of feasible actions $\mu(\mathsf{S}(u_{\leq t}))$ is differentiable and strictly logarithmically-concave with respect to the subsequence of actions $u_t=(u_1,\ldots,u_t)$ for all $t\in [\nt]$, which is also common in practice, e.g., it holds in the case of inventory constraints $\sum_{t=1}^{\nt}\|u_t\|_2\leq B$ with a budget $B>0$.  Finally, recall the definition of the set of subsequent feasible action trajectories:
\begin{align*}
\mathsf{S}(u_{\leq t})
:=& \Big\{v_{>t}\in\mathsf{U}^{\nt-t}: v\text{ satisfies } \eqref{eq:offline_2}-\eqref{eq:offline_3}, v_{\leq t} = u_{\leq t} \Big\}.
\end{align*}
Putting the above together, we can state our assumption formally as follows.

\begin{assumption}
\label{ass:2}
The cost functions $c_t(u):\mathbb{R}\rightarrow \mathbb{R}_+$ are differentiable and strictly convex. The mappings $\mu(\mathsf{S}(u_{\leq t})):\mathbb{R}^{t}\rightarrow\mathbbm{R}_+$ are differentiable and strictly logarithmically-concave.
\end{assumption}

Given regularity of the cost functions and time-varying constraints, we can prove optimality of PPC.

\begin{theorem}[Existence of optimal actions]
\label{thm:optimality}
Let $\mathsf{U}=\mathbb{R}$. Under Assumption~\ref{ass:1} and \ref{ass:2}, there exists a sequence $b_1,\ldots,b_{\nt}$ such that implementing~\eqref{eq:PPC} with $\beta_t=b_t$ and $p_t=p_t^*$ at each time $t\in [\nt]$ ensures $u=(u_1^*,\ldots,u_\nt^*)$, i.e., the generated actions are optimal.
\end{theorem}

Crucially, Theorem~\ref{thm:optimality} shows that there exists a sequence of ``good'' tuning parameters so that the PPC scheme is able to generate optimal actions under reasonable assumptions. However, note that the assumption of $\mathsf{U}=\mathbb{R}$ is fundamental.  When the action space $\mathsf{U}$ is discrete or $\mathsf{U}$ is a high-dimensional space, it is impossible to generate the optimal actions because, in general, fixing $t$, the differential equations in the proof of Theorem~\ref{thm:optimality} (see Appendix~\ref{app:proof_optimality}) do not have the same solution for all $\beta_t>0$. Therefore a detailed regret analysis is necessary in such cases, which is a challenging task  for future work. 

}

\section{Numerical Results}
\label{sec:experiment}

In this section, we present experimental results for the case of online EV charging, introduced  in Section~\ref{sec:example} as an example of our system model (see Section~\ref{sec:model}). {The notation used in this section, if not defined, can be found in   Section~\ref{sec:example}.}

{
\subsection{Experimental setups}
\label{sec:experiemental_results}

In the following, we present settings of parameters and useful metrics in our experiments. 

{
\subsubsection{Dataset and hardware}

We use real EV charging data from ACN-Data~\cite{lee2019acn}, which is a dataset collected from adaptive EV charging networks (ACNs) at Caltech and JPL. The detailed hardware setup for that EV charging network structure can be found in~\cite{lee2020adaptive}.}

\subsubsection{Error Metrics}
Recall the EV charging example in optimization~\eqref{eq:ev_1}-\eqref{eq:ev_2}. We first introduce two error metrics to measure the EV charging constraint violations. Note that the constraints~\eqref{eq:f1},~\eqref{eq:f5} and ~\eqref{eq:f4} are hard constraints depending only on the scheduling policy, but not the actions and energy demands. Therefore they can be automatically satisfied in our experiments by fixing a scheduling policy satisfying them such as least laxity first. Violations may happen on constraint~\eqref{eq:f3} and~\eqref{eq:f2}. To measure the violation of~\eqref{eq:f3},
we use the (normalized) mean squared error (MSE) as the tracking error:
\begin{align}
\label{eq:mse}
\mathsf{MSE}:= \sum_{k=1}^{L}\sum_{t=1}^{\nt}\Big|\sum_{j=1}^{N}s_{t}^{(k)}(j)-u^{(k)}_t\Big|^2 / \left(L\times \nt\times\xi\right),
\end{align}
where $u^{(k)}_t$ is the $t$-th power signal for the $k$-th test and $s_{t}^{(k)}(j)$ is the energy scheduled to the $j$-th charging session at time $t$ for the $k$-th test. {To better approximate real-world cases, we consider an additional \textit{operational constraints} for the operator (central controller) and require that $u_t\leq \xi$ (kWh) for every $t\in [\nt]$.} The total number of tests is $L$ and the total number of charging sessions is $\nn$. Additionally, define the mean percentage error with respect to the undelivered energy corresponding to~\eqref{eq:f2}
as
\begin{align}
\label{eq:mpe}
\mathsf{MPE}:=1- \sum_{k=1}^{L}\sum_{t=1}^{\nt}\sum_{j=1}^{\nn}s_{t}^{(k)}(j)\big/{ \Big((L\times \nt)\cdot\sum_{j=1}^{\nn}e_j\Big)},
\end{align}
where $e_j$ is the energy request for each charging session $j\in [\nn]$; $s_{t}^{(k)}(j)$ is the energy scheduled to the $j$-th charging session at time $t$ for the $k$-th test.

\subsubsection{Hyper-parameters}

{
\label{sec:details}
\begin{table}[h]
    \centering
       \caption{Hyper-parameters in the experiments.}
    \begin{tabular}{l|l}
    \hline
    Parameter & Value\\
   \hline
   \hline
     \multicolumn{2}{c}{System Operator} \\
     \hline
     Number of power levels $|\mathsf{U}|$ & $10$\\
     Cost functions $c_1,\ldots,c_{\nt}$ & Average LMPs\\ 
     Operator function $\phi$ & Penalized Predictive Control\\
     Tuning parameter $\beta$ & $1\times 10^3$ - $1\times 10^6$\\
    \hline
     \multicolumn{2}{c}{EV Charging Aggregator} \\
     \hline
      Number of Chargers ${\ns}$  & 54\\
      State space $\mathsf{X}$ &$\mathbb{R}_{+}^{108}$ \\
      Action space & $[0,1]^{10}$\\
      Time interval $\Delta$ & $12$ minutes\\
      Private vector $(a(j),d(j),e(j),r(j))$ & ACN-Data~\cite{lee2019acn}\\
      Power rating & $150$ kW\\
      Scheduling algorithm $\pi$ & Least Laxity First (LLF)\\
      Laxity & $d_t(j)-e_t(j)/r(j)$\\
      RL algorithm & Soft Actor-Critic (SAC)~\cite{haarnoja2018soft}\\
      Optimizer & Adam~\cite{kingma2014adam}\\
      Learning rate & $3\cdot 10^{-4}$\\
      Discount factor & $0.5$\\
      Relay buffer size & $10^6$\\
      Number of hidden layers & $2$\\
      Number of hidden units per layer & $256$\\
      Number of samples per minibatch & $256$\\
      Non-linearity & ReLU\\
      Reward function & $\sigma_1=0.1$, $\sigma_2=0.2$, $\sigma_3=2$\\
      Temperature parameter & $0.5$\\
       \hline
      \hline
    \end{tabular}
  \label{table:parameter}
\end{table}

The detailed parameters used in our experiments are shown in Table~\ref{table:parameter}. 

\paragraph{Control spaces}
For the experimental results presented in this section, the control state space is $\mathsf{X} = \mathbbm{R}_{+}^{2\times{\ns}}$ where ${\ns}$ is the total number of charging stations and a state vector for each charging station is $(e_t,[d(j)-t]^{+})$, i.e., the remaining energy to be charged and the remaining charging time if it is being used (see Section~\ref{sec:example}); otherwise the vector is an all-zero vector. The control action space is $\mathsf{U} = \{0,15,30,\ldots,150\}$ (unit: kW) with $|\mathsf{U}|=10$, unless explicitly stated. The scheduling policy $\pi$ is fixed to be least-laxity-first (LLF).

\paragraph{RL spaces}
The RL action space\footnote{Note that the RL action space (consisting of $p_{t}$'s) and state space  (consisting of $x_t$'s) referred here are the standard definitions in the context of RL and they are different from the ``control action space" $\mathsf{U}$ and ``control state space" $\mathsf{X}$ defined in Section~\ref{sec:model}.} of the Markov decision process used in the RL algorithm is $[0,1]^{10}$. The outputs of the neural networks are clipped into the probability simplex (space of MEF) $\mathsf{P}$  afterwards.

\paragraph{RL rewards}
We use the following specific reward function for our EV charging scenario, as a concrete example of~\eqref{eq:reward}:
\begin{align}
\nonumber
r_{\mathrm{EV}}(\overline{x}_{t},p_t) =&\mathbbm{H}(p_t)\\
\nonumber
+ &\sigma_1 \sum_{i=1}^{\nn'}\left\|u_t(i)\right\|_2\\
\nonumber
-&\sigma_2 \sum_{i=1}^{\nn'}\left(\mathbf{I}(a(j_i)\leq t\leq a(j_i)+\Delta)\Big[e(i) - \sum_{t=1}^{\nt}u_t(i) \Big]_{+}\right)\\
-&\sigma_3\left|\phi_t(p_t) - \sum_{i=1}^{\nn'}u_t(j)\right|
\label{eq:reward_EV}
\end{align}
where $\sigma_1,\sigma_2$ and $\sigma_3$ are positive constants; $\nn'$ is the number of EVs being charged; $\phi_t$ is the operator function, which is specified by~\eqref{eq:PPC}; $\mathbf{I}(\cdot)$ denotes an indicator function and $a(j_i)$ is the arrival time of the $i-th$ EV in the charging station with $j_i$ being the index of this EV in the total accepted charging sessions $[\nn]$. The entropy function $\mathbbm{H}(p_t)$ in the first term is a greedy approximation of the definition of MEF (see Definition~\ref{def:MEF}).
The second term is to further enhance charging performance and the last two terms are realizations of the last term in~\eqref{eq:reward} for constraints~\eqref{eq:f3} and~\eqref{eq:f2}. Note that The other constraints in the example shown in Section~\ref{sec:example} can automatically be satisfied by enforcing the constraints in the fixed scheduling algorithm $\pi$.
With the settings described above, in Figure~\ref{fig:reward} we show a typical training curve of the reward function in~\eqref{eq:reward_EV}.  We observe policy convergence with respect to a wide range of choices of the hyper-parameters $\sigma_1,\sigma_2$ and $\sigma_3$. In our experiments, we do not optimize them but fix the constants in~\eqref{eq:reward_EV} as $\sigma_1=0.1$, $\sigma_2=0.2$ and $\sigma_3=2$.

\paragraph{Cost functions}

We consider the specific form of costs in~\eqref{eq:ev_1}.
In the RL training process, we train an aggregator function $\psi$ using
linear price functions $c_t=1-t/24$ where $t\in [0,24]$ (unit: Hrs)  is the time index and we test the trained system with real price functions $c_1,\ldots,c_{\nt}$ being the average locational marginal prices (LMPs) on the CAISO (California Independent System Operator) day-ahead market in 2016 (depicted at the bottom of Figure~\ref{fig:charging}).

\paragraph{Tuning parameters}

In PPC defined in Algorithm~\ref{alg:PPC}, there is a sequence of tuning parameters $(\beta_1,\ldots,\beta_{\nt})$. In our experiments, we fix $\beta_t=\beta$ for all $t\in [\nt]$ where $\beta>0$ is a universal tuning parameter that can be varied in our experiments. 

}

\subsection{Experimental results}

\begin{figure*}[t]
    \centering
    \includegraphics[width=0.85\linewidth]{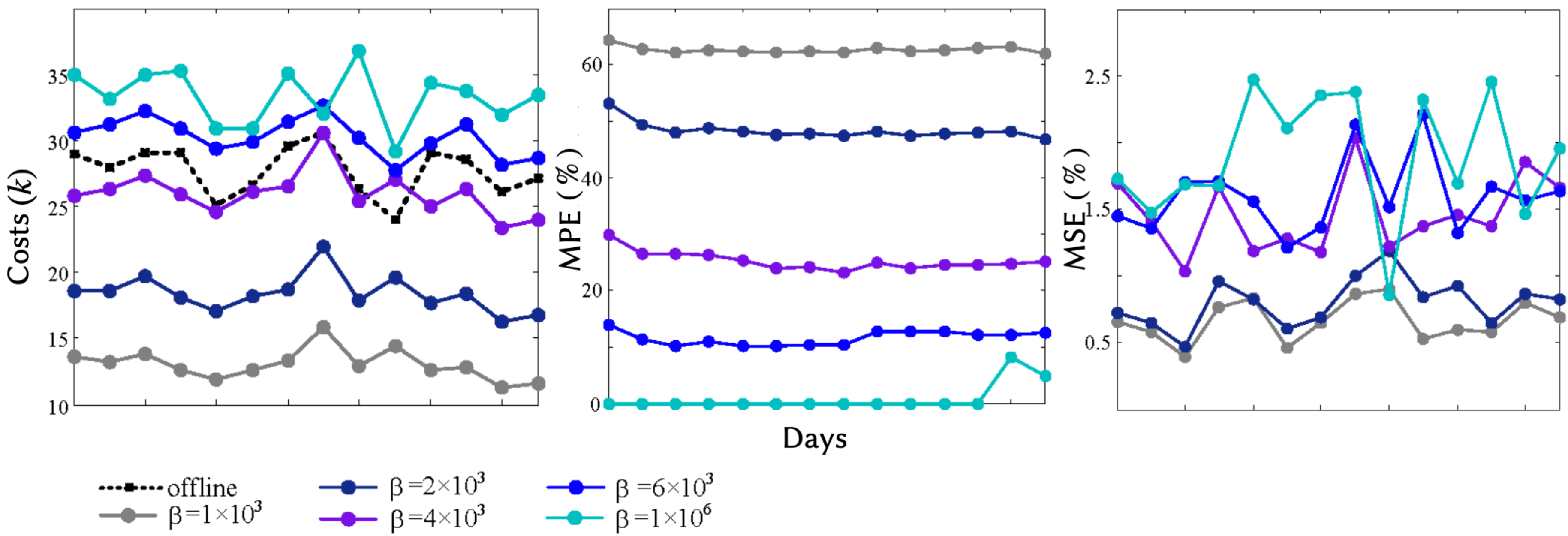}
    \caption{Trade-offs of cost and charging performance.  The dashed curve in the left figure corresponds to offline optimal cost. The tested days are selected (with no less than $30$ charging sessions, i.e., $N\geq 30$) from Dec. 2, 2019 to Jan. 1, 2020.}
    \label{fig:tuning}
\end{figure*}

\begin{figure*}[h!]
    \centering
    \includegraphics[scale=0.45]{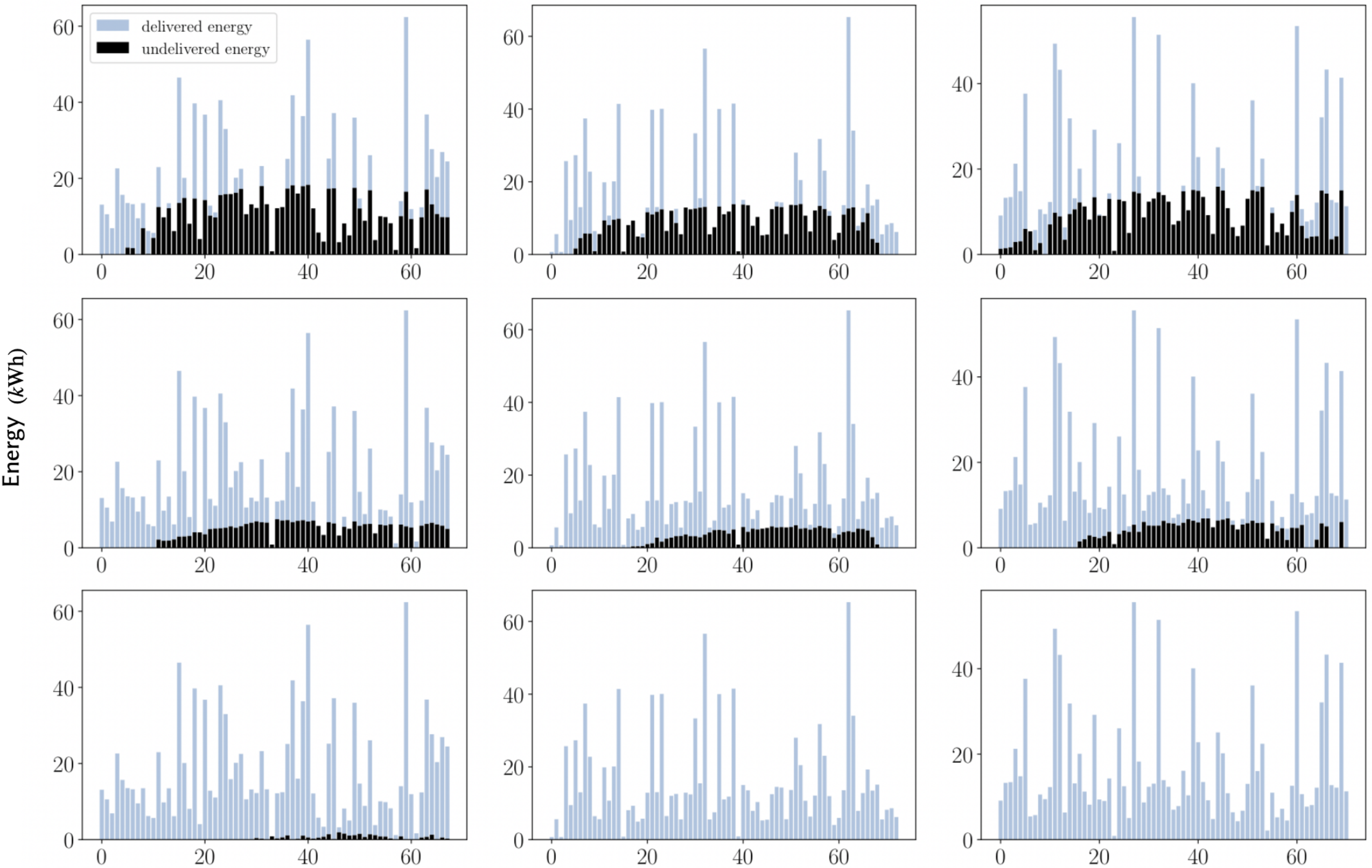}
    \caption{Charging results of EVs controlled by PPC with tuning parameters $\beta=2\times 10^3$ (top), $4\times 10^3$ (mid) and $6\times 10^3$ (bot) for selected days (with no less than $30$ charging sessions, i.e., $N\geq 30$) from Dec. 2, 2019 to Jan. 1, 2020. {Each bar represents a charging session.}}
    \label{fig:battery}
\end{figure*}
\subsubsection{Sensitivity of $\beta$}
We first show how the changes of the tuning parameter $\beta$ affect the total cost and feasibility. Figure~\ref{fig:tuning} compares the results by varying $\beta$. The agents are trained on data collected from Nov. 1, 2018 to Dec. 1, 2019 and the tests are performed on data from Dec. 2, 2019 to Jan. 1, 2020 .  Weekends and days with less than $30$ charging sessions are removed from both training and testing data.  For charging performance, we show in Figure~\ref{fig:battery} the battery states of each session after the charging cycle ends, tested with tuning parameters $\beta=2\times 10^3,4\times 10^3$ and $6\times 10^3$ respectively. The results indicate that with a sufficiently large tuning parameter, the charging actions given by the PPC is able to satisfy EVs' charging demands and in  practice, there is a trade-off between costs and feasibility depending on the choice of tuning parameters.

\begin{figure}[t]
    \centering
    \includegraphics[scale=0.6]{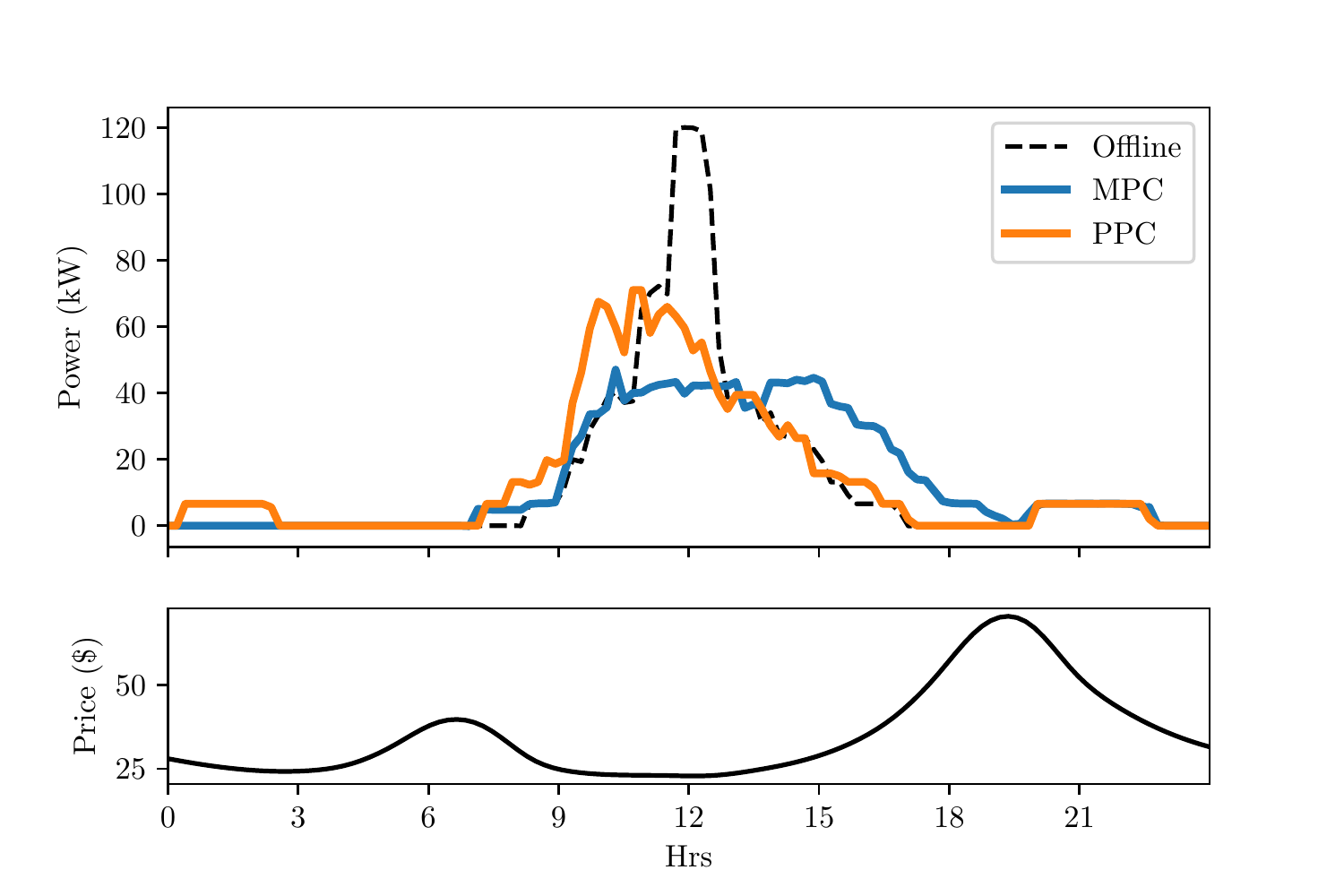}
    \caption{{ Substation charging rates generated by the PPC (orange) in the closed-loop control shown in Algorithm~\ref{alg:online_control}, together with the MPC generated (blue) and global optimal (dashed black) charging rates. }}
    \label{fig:charging}
\end{figure}

\subsubsection{Charging curves}
{


In Figure~\ref{fig:charging}, substation charging rates (in kW) are shown. The charging rates generated by the PPC correspond to a trajectory $(\sum_{j}{s}_1(j)/\Delta,\ldots,\sum_{j}{s}_{\nt}(j))\Delta$), which is the aggregate charging power given by the PPC for all EVs at each time $t=1,\ldots,\nt$. 
The agent is trained on data collected at Caltech from Nov. 1, 2018 to Dec. 1, 2019 and tested on Dec. 16, 2019 at Caltech using real LMPs on the CAISO day-ahead market in 2016.  We use a tuning parameter $\beta=4\times 10^3$ for both training and testing. The figure highlights that, with a suitable choice of tuning parameter, the operator is able to schedule charging at time slots where prices are lower and avoid charging at the peak of prices, as desired. In particular, it achieves a  lower cost compared with the commonly used MPC scheme described in~\eqref{eq:MPC}-\eqref{eq:MPC2}.The offline optimal charging rates are also provided.}

\begin{figure}[t]
    \centering    \includegraphics[scale=0.4]{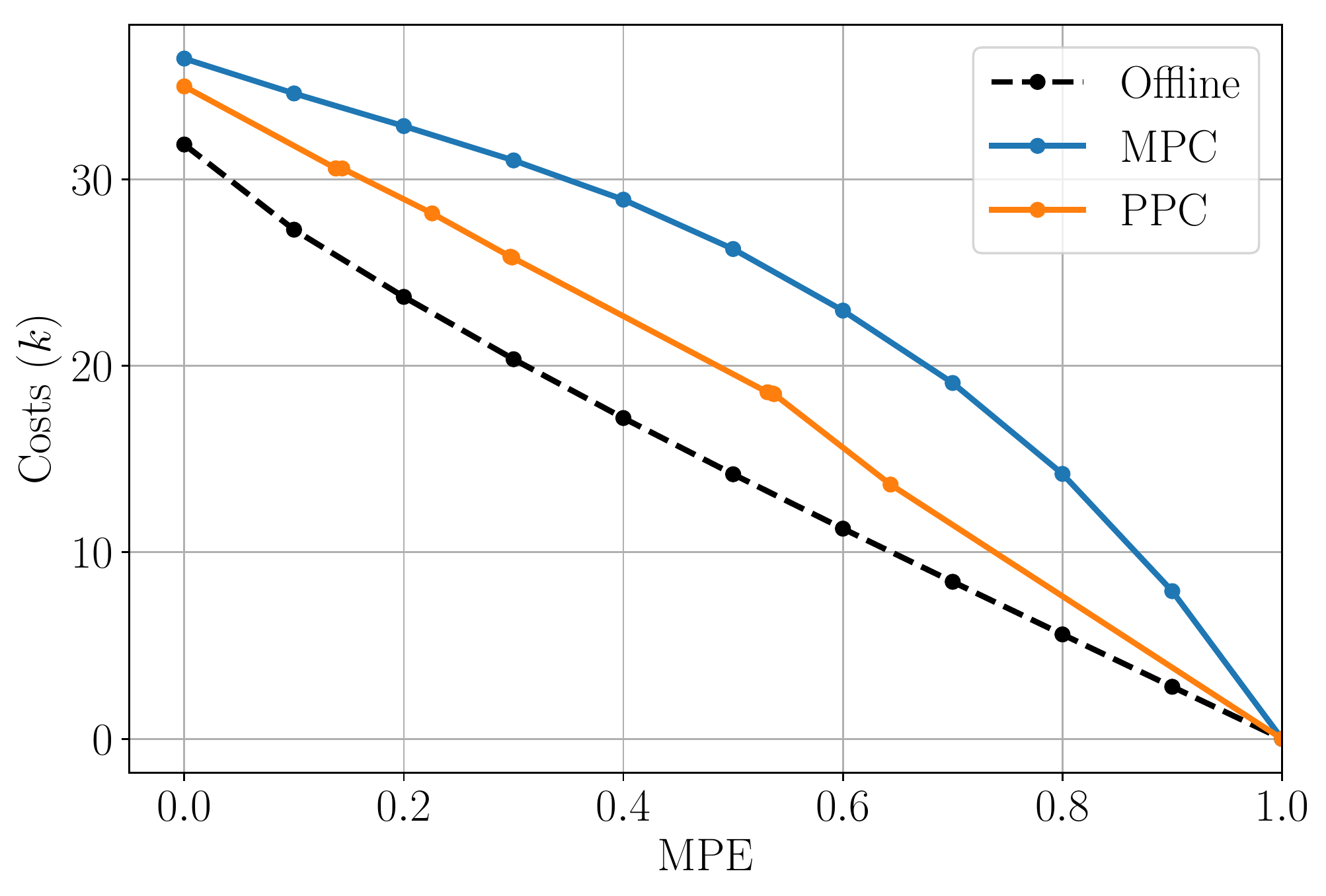}
          \caption{Cost-energy curves for the offline optimization in~\eqref{eq:offline_1}-\eqref{eq:offline_3} (for the example in Section~\ref{sec:example}), MPC (defined in~\eqref{eq:MPC}-\eqref{eq:MPC2}) and PPC (introduced in Section~\ref{sec:PPC}).}
\label{fig:cost_energy}
\end{figure}

\subsubsection{Comparison of PPC and MPC}
In Figure~\ref{fig:cost_energy} we show the changes of the cumulative costs by varying the mean percentage error (MPE) with respect to the undelivered energy defined in~\eqref{eq:mpe}. There are in total $K=14$ episodes tested for days selected from Dec. 2, 2019 to Jan. 1, 2020 (days with less than $30$ charging sessions are removed, i.e. we require, $N\geq 30$). Note that $0\leq \mathsf{MPE}\leq 1$ and the larger MPE is, the higher level of constraint violations we observe.
We allow constraint violations and modify parameters in the MPC and PPC to obtain varying MPE values. For the PPC, we vary the tuning parameter $\beta$ to obtain the corresponding costs and MPE. For the MPC in our tests, we solve the following optimization at each time for obtaining the charging decisions $s_t=(s_t(1),\ldots,s_{t}(\nn'))$:
\begin{subequations}
\begin{align}
\label{eq:MPC}
s_t = \argmin_{s_t} \sum_{\tau=t}^{t'}c_{\tau}\Big(\sum_{i=1}^{\nn'}s_{\tau}(i)\Big) \ &
\mathrm{subject} \ \mathrm{to}:\\
s_{\tau}(i)= 0 \ , \  \tau<a(i), & \ i=1,\ldots,\nn',\\
s_{\tau}(i)= 0\ , \ \tau>d(i), & \  i=1,\ldots,\nn',\\
\sum_{i=1}^{\nn'}s_{\tau}(i)= u_{t}, & \ \tau=t,\ldots,t',\\
\sum_{\tau=1}^{\nt}s_{\tau}(i) = \gamma \cdot e(i), & \ i=1,\ldots,\nn',\\
\label{eq:MPC2}
0\leq s_{\tau}(i)\leq r(i), & \ \tau=t,\ldots,t'
\end{align}
\end{subequations}
where at time $t$, the integer $\nn'$ denotes the number of EVs being charged at the charging station and the time horizon of the online optimization is from $\tau=t$ to $t'$, which is the latest departure time of the present charging sessions; $a(i)$ and $d(i)$ are the arrival time and departure time of the $i$-th session; $\gamma>0$ relaxes the energy demand constraints and therefore changes the MPE region for MPC. The offline cost-energy curve is obtained by varying the energy demand constraints in~\eqref{eq:f2} in a similar way. We assume there is no admission control and an arriving EV will take a charger whenever it is idle for both MPC and PPC. Note that this MPC framework is widely studied~\cite{rosolia2018data} and used in EV charging applications~\cite{lee2018large}. It requires the precise knowledge of a $108$-dimensional state vector of $54$ chargers at each time step. 
We observe that with \textit{only} feasibility information, PPC outperforms MPC for all $0\leq \mathsf{MPE}\leq 1$. The main reason that PPC outperforms vanilla MPC is that PPC utilizes MEF as its input, which is generated by a pre-trained aggregator function. Therefore the MEF may contain useful future feasibility information that vanilla MPC does not know, despite that it is trained and tested on separate datasets.}


\section{Concluding remarks}
\label{sec:conclusion}
This paper formalizes and studies the closed-loop control framework created by the interaction between a system operator and an aggregator.  Our focus is on the feedback signal provided by the aggregator to the operator that summarizes the real-time availability of flexibility among the loads controlled by the aggregator.  We present the design of an maximum entropy feedback (MEF) signal based on entropic maximization.
We prove a close connection between the MEF signal and the system capacity, and show that when the signal is used the system operator can perform online cost minimization while provably respecting the private constraints of the loads controlled by the aggregator {and satisfying optimality under certain regularity assumptions.} Further, we illustrate the effectiveness of these designs using simulation experiments of an EV charging facility.  

There is much left to explore about this MEF signal presented in this work.  In particular, computing it is computationally intensive and we use reinforcement learning for approximating the MEF.  Improving the learning design and developing other approximations are of particular interest.  Further, exploring the use of flexibility feedback for operational objectives beyond cost minimization and capacity estimation is an important goal.  Finally, exploring the application of the defined real-time aggregate flexibility in other settings, such as multi-aggregator systems, frequency regulation and real-time pricing, is exciting.

\addcontentsline{toc}{section}{Bibliography}
{\bibliographystyle{IEEEtran}
	\small
	\bibliography{final_lib}}

\begin{figure*}
    \centering
    {
    \begin{align}
    \label{eq:sc1}
    u_t^* =   \argmin_{u_t\in\mathsf{U}} \Bigg(c_t(u_t)  + \underbrace{\min_{u_{t+1:\nt}} \left(\sum_{\tau=t+1}^{\nt}c(u_{\tau}) -\log {\mu\left(\mathsf{S}(u^*_{t-k:t-1},u_{t:\nt})\right)} \right)}_{f(u_t)}\Bigg)
    .
    \end{align}}
    \hrule
\end{figure*}

\appendices

{
\section{Proof of Lemma~\ref{lemma:measurability}}
\label{app:measurability}

\begin{proof}
We first define a set $f^{-1}\left(\mathsf{X}_t\right)$ denoting the inverse image of the set $\mathsf{X}_t$ for actions:
$
 f^{-1}\left(\mathsf{X}_t\right)({u}_{<t}):=\left\{u\in\mathsf{U}: f(x_{t-1},u)\in \mathsf{X}_t \right\}.  
$
The inverse image $f^{-1}\left(\mathsf{X}_t\right)$ depends only on the past actions ${u}_{<t}$ since the states ${x}_{<t}$ are determined by ${u}_{<t}$ and a pre-fixed initial state $x_1$ via the dynamics in~\eqref{eq:system}.
Note that $\mathsf{X}_t$ and the dynamic $f$ are Borel measurable. Therefore the inverse image $f^{-1}\left(\mathsf{X}_t\right)$ is also a Borel set, implying that the intersection $\mathsf{U}_t\bigcap f^{-1}\left(\mathsf{X}_t\right)$ is also Borel measurable. The set of  feasible action trajectories $\mathsf{S}$ can be reprised as
\begin{align*}
\mathsf{S}:=\left\{{u}\in\mathsf{U}^{\nt}:u_t\in\mathsf{U}_t\bigcap f^{-1}\left(\mathsf{X}_t\right)(u_{<t}), \forall t\in [\nt]\right\},
\end{align*}
which is a Borel measurable set of all \textit{feasible} sequences of actions.
\end{proof}}

\section{Proof of Lemma~\ref{lemma:explicit}}
\label{app:proof_MEF}
\begin{proof}[Proof of Lemma~\ref{lemma:explicit}] 
We prove the statement by induction. It is straightforward to verify the results hold when $\nt=1$. We suppose the lemma is true when $\nt=m$. Suppose $\nt=m+1$.
Let
\begin{align*}
\digamma(u):=\max_{p_2,\ldots,p_{\nt}}\sum_{t=2}^{\nt}\mathbb{H}\left(U_{t}|\mathbf{U}_{2:t-1}; U_1=u\right)
\end{align*}
denote the optimal value corresponding to the time horizon $t\in [\nt]$, given the first action $U_1 = u$. By the definition of conditional entropy, we have
\begin{align*}
{\digamma}= \max_{p_{1}}\int_{u\in\mathsf{U}}p_{1}(u)\digamma(u)\mathsf{d}u+\mathbb{H}(p_{1}).
\end{align*}
By the induction hypothesis, $\digamma(u)= \mu\left(\mathsf{S}(u)\right)$. Therefore, 
\begin{align*}
{\digamma} =& \max_{p_{1}}\int_{u\in\mathsf{U}}p_{1}(u)\log\mu\left(\mathsf{S}(u)\right)\mathsf{d}u+\mathbb{H}(p_{1})\\
=&\max_{p_{1}}\int_{u\in\mathsf{U}}p_{1}(u)\log\left(\frac{\mu\left(\mathsf{S}(u)\right)}{p_1(u)}\right)\mathsf{d}u
\end{align*}
whose optimizer $p_{1}^*$ satisfies (\ref{eq:optimal}) and we get ${\digamma}=\mu\left(\mathsf{S}\right)$. The lemma follows by finding the optimal conditional distributions $p_1^*,\ldots,p_{\nt}^*$ inductively.
\end{proof}

\section{Proof of Corollary~\ref{coro:property}}
\label{app:proof_feasibility}

\begin{proof}[Proof of Corollary~\ref{coro:property}]
Lemma~\ref{lemma:explicit} shows that the value of the density function corresponding to choosing $u_t=u$ in the MEF is proportional to the measure of $\mathsf{S}((u_{<t}, u))$, completing the proof of interpretability. According to the explicit expression in~\eqref{eq:optimal} of the MEF, the selected action $u$ always ensures that $\mu(\left(\mathsf{S}((u_{< t},u)\right))>0$ and therefore the set $\mathsf{S}((u_{< t},u)$ is non-empty. This guarantees that the generated sequence $u$ is always in $\mathsf{S}$.
\end{proof}

\section{Proof of Corollary~\ref{corollary:feasibility}}

\begin{proof}[Proof of Corollary~\ref{corollary:feasibility}]
The explicit expression in Lemma~\ref{lemma:explicit} ensures that whenever $p^*_t(u_t|u_{<t})>0$, then there is always a feasible sequence of actions in $\mathsf{S}\left(u_{<t}\right)$. Now, if the tuning parameter $\beta_t>0$, then the optimization~\eqref{eq:PPC} guarantees that $p^*_t(u_t|u_{<t})>0$ for all $t\in [\nt]$; otherwise, the objective value in~\eqref{eq:PPC} is unbounded. Corollary~\ref{coro:property} guarantees that for any sequence of actions $u=(u_1,\ldots,u_{\nt})$, if $ p^*_t(u_t|u_{<t})>0$ for all $t\in [\nt]$, then $u\in\mathsf{S}$. Therefore, the sequence of actions $u$ given by the PPC is always feasible. 
\end{proof}

{

\section{Proof of Theorem~\ref{lemma:connection}}
\label{app:proof_connection}

\begin{proof}[Proof of Lemma~\ref{lemma:connection}]
We note that the offline optimization~\eqref{eq:offline_1}-\eqref{eq:offline_3} is equivalent to
\begin{align}
\label{eq:2.1}
\inf_{u\in\mathsf{U}^{\nt}} \sum_{t=1}^{\nt} &c_t(u_t) - \beta\log q(u)
\end{align}
for any $\beta>0$ and $q(u)$ is a uniform distribution on $\mathsf{S}$:
\begin{align}
\label{eq:uniform}
    q(u):=\begin{cases}
    1/\mu\left(\mathsf{S}\right) & \text{ if } u\in \mathsf{S}\\
    0 & \text{ otherwise}
    \end{cases}
\end{align}
where $\mu(\cdot)$ is the Lebesgue measure.
Further, decomposing the joint distribution $q(u)=\prod_{t=1}^{\nt}p^*_t(u_t|u_{<t})$ into the conditional distributions given by~\eqref{eq:af1}-\eqref{eq:af3}, the objective function~\eqref{eq:2.1} becomes
\begin{align}
\nonumber
&\sum_{t=1}^{\nt} c_t(u_t) - \beta\log \left(\prod_{t=1}^{\nt} p^*_t(u_t|u_{<t})\right)\\
\nonumber
=&\sum_{t=1}^{\nt} \left(c_t(u_t) - \beta\log p^*_t(u_t|u_{<t})\right),
\end{align}
which implies the lemma.
\end{proof}

\section{Proof of Theorem~\ref{thm:optimality}}

\label{app:proof_optimality}
\begin{proof}
Define the following optimal \textit{cost-to-go function}, which computes the minimal cost given a subsequence of actions:
\begin{align*}
&V^{\mathsf{OPT}}_t(u_{<t})
:= \min_{u_{t:\nt}\in\mathsf{U}^{\nt-t+1}} \left(\sum_{\tau=t}^{\nt}c(u_{\tau}) -\beta\log p^*_t(u_{t:\nt}|u^*_{t-k:t-1}) \right)\\
=& \min_{u_t\in\mathsf{U}} \Bigg(c(u_t) -\log p^*_t(u_t|u^*_{t-k:t-1})+ \\
&\min_{u_{t+1:\nt}\in\mathsf{U}^{\nt-t}} \Big(\sum_{\tau=t+1}^{\nt}c(u_{\tau}) -\log p^*_{t+1}(u_{t+1:\nt}|u^*_{t-k:t}) \Big)\Bigg)\\
=& \min_{u_t\in\mathsf{U}} \left(c(u_t) -\log p^*_t(u_t|u^*_{t-k:t-1}) + V^{\mathsf{OPT}}_{t+1}(u_{\leq t}) \right).
\end{align*}
Let $\mu(
\cdot)$ denote the Lebesgue measure. Based on the definition of the optimal cost-to-go functions defined above and applying Lemma~\ref{lemma:connection}, we obtain the following expression of the optimal action $u_t^*$ at each time $t\in [\nt]$:
\begin{align*}
u_t^* =& \argmin_{u_t\in\mathsf{U}} \left(c_t(u_t) -\beta\log p^*_t(u_t|u^*_{t-k:t-1}) + V^{\mathsf{OPT}}_{t+1}\left(u_{t-k+1:t}\right)\right)\\
= & \argmin_{u_t\in\mathsf{U}} \Bigg(c_t(u_t) -\beta\log \frac{\mu\left(\mathsf{S}(u^*_{t-k:t-1},u_t)\right)}{\mu\left(\mathsf{S}(u^*_{t-k:t-1})\right)} \\
&+ \min_{u_{t+1:\nt}} \Big(\sum_{\tau=t+1}^{\nt}\big(c(u_{\tau})
-\beta\log \frac{\mu\left(\mathsf{S}(u^*_{t-k:t-1},u_{t:\tau})\right)}{\mu\left(\mathsf{S}(u^*_{t-k:t-1},u_{t:\tau-1})\right)} \big)\Big)\Bigg)\\
= & \argmin_{u_t\in\mathsf{U}} \Bigg(c_t(u_t) -\beta\log \frac{\mu\left(\mathsf{S}(u^*_{t-k:t-1},u_t)\right)}{\mu\left(\mathsf{S}(u^*_{t-k:t-1})\right)}\\
&+ \min_{u_{t+1:\nt}} \Big(\sum_{\tau=t+1}^{\nt}c(u_{\tau}) +\beta \log \frac{\mu\left(\mathsf{S}(u^*_{t-k:t-1},u_t)\right)}{\mu\left(\mathsf{S}(u^*_{t-k:t-1},u_{t:\nt})\right)} \Big)\Bigg),
\end{align*}
which implies~\eqref{eq:sc1} and when $u_{<t}=u^*_{<t}$, the solution of the PPC in Algorithm~\ref{alg:PPC} satisfies
\begin{align}
\nonumber
u_t =
&\argmin_{u_t\in \mathsf{U}} \left(c_t(u_t) - \beta\log p_t(u_t|u^*_{t-k:t-1})\right)\\
\nonumber
    = &\argmin_{u_t\in \mathsf{U}} \left(c_t(u_t) -  \beta\log \frac{\mu\left(\mathsf{S}(u^*_{t-k:t-1},u_t)\right)}{\mu\left(\mathsf{S}(u^*_{t-k:t-1})\right)}\right)\\
    \label{eq:function_g}
    = &\argmin_{u_t\in \mathsf{U}} \Bigg(c_t(u_t) + \beta \underbrace{\log {\left(1/\mu\left(\mathsf{S}(u^*_{t-k:t-1},u_t)\right)\right)}}_{g(u_t)}\Bigg)
    .
\end{align}
Since the cost functions $c_t(u)$ and the measure $\log(1/\mu(\mathsf{S}(u)))$ are strictly convex, the inner minimization in~\eqref{eq:sc1} is a convex minimization and hence $f(u_t)$ is convex. Therefore, $u_t^*$ in~\eqref{eq:sc1} is unique. Denoting by $c'$ and $f'$ the corresponding derivatives of a given cost function $c$ and the function $f$ defined in~\eqref{eq:sc1},  we have
\begin{align*}
    c_t'(u_t^*) + f'(u_t^*) = 0.
\end{align*}
Furthermore, the unique solution of the PPC scheme satisfies
\begin{align*}
    c_t'(u_t) + \beta g'(u_t) = 0
\end{align*}
where $g'$ is the derivative of the function $g$ defined in~\eqref{eq:function_g}.
Choosing $\beta=b_t=f'(u_t^*)/g'(u_t^*)$ implies that $u_t=u_t^*$ for all $t\in [\nt]$.
\end{proof}

}

\end{document}